\documentclass[12pt]{amsart}

\usepackage{amsfonts,latexsym,rawfonts,amsmath,amssymb,amsthm,mathrsfs}
\usepackage[colorlinks=true,urlcolor=blue,linkcolor=blue,anchorcolor=blue,citecolor=blue]{hyperref}
\usepackage{graphicx}
\usepackage{comment}
\usepackage[]{caption2}

\usepackage[pagewise]{lineno}
\usepackage{amsthm}
\usepackage{lipsum}
\usepackage[raggedright]{sidecap}\sidecaptionvpos{figure}{t}
\usepackage{tikz}
\usepackage{amssymb}
\usepackage{enumerate}

\allowdisplaybreaks[4]

\makeatletter
\@namedef{subjclassname@2020}{\textup{2020} Mathematics Subject Classification}
\makeatother

\numberwithin{equation}{section}

\RequirePackage{color}
\theoremstyle{plain}

\newtheorem{theorem}{Theorem}[section]
\newtheorem{lemma}{Lemma}[section]
\newtheorem{corollary}{Corollary}[section]
\newtheorem{proposition}{Proposition}[section]

\newtheorem{remark}{Remark}[section]
\def\R{\mathbb R}

\begin{document}
\title[On rigidity of hypersurfaces with constant shifted curvature functions]{On rigidity of hypersurfaces with constant shifted curvature functions in warped product manifolds}

\author{Weimin Sheng}
\address{Weimin Sheng: School of Mathematical Sciences, Zhejiang University, Hangzhou 310058, China.}
\email{shengweimin@zju.edu.cn}

\author{Yinhang Wang}
\address{Yinhang Wang: School of Mathematical Sciences, Zhejiang University, Hangzhou 310058, China.}
\email{wangyinhang@zju.edu.cn}

\author{Jie Wu}
\address{Jie Wu: School of Mathematical Sciences, Zhejiang University, Hangzhou 310058, China.}
\email{wujiewj@zju.edu.cn}


\begin{abstract}
In this paper, we give some new characterizations of umbilic hypersurfaces in general warped product manifolds, which can be viewed as generalizations of the work in \cite{KLP18} and \cite{WX14}. Firstly, we prove the rigidity for hypersurfaces with constant linear combinations of shifted higher order mean curvatures. Using integral inequalities and Minkowski-type formulas, we then derive rigidity theorems in sub-static warped product manifolds, including cases that the hypersurface satisfies some nonlinear curvature conditions. Finally, we show that our results can be applied to more general warped product manifolds, including the cases with non-constant sectional curvature fiber.
\end{abstract}

\maketitle

\baselineskip16pt
\parskip3pt

\section{Introduction}

The rigidity problem for hypersurfaces in Riemannian manifolds has been studied widely. It is well-known that Alexandrov \cite{A56} proved that any closed, embedded hypersurface in $\mathbb{R}^{n+1}$ with constant mean curvature must be a round sphere. This result was established by the reflection method and later reproved by Reilly's integral formula \cite{R77}.
In the 1990s, A. Ros \cite{Ros88,Ros87} generalized the result to hypersurfaces with constant higher order mean curvatures, while Montiel and Ros \cite{MR91} provided an alternative proof, inspired by the Heintze-Karcher inequality \cite{HK78}, using a direct way based on integration. Later, Montiel \cite{M99} studied hypersurfaces with constant mean curvature in warped product manifolds, albeit under the assumption the hypersurface being star-shaped. This restriction was later removed by Brendle \cite{B13}, who proved an Alexandrov-type theorem in a class of sub-static warped product manifolds, including de Sitter-Schwarzschild and Reissner-Nordstrom manifolds. Brendle and Eichmair \cite{BE13} extended these results to the closed star-shaped hypersurfaces with constant higher order mean curvatures. Other generalizations can be found in  \cite{ADM13,AD14,BC97,HLMG09,HMZ01,K98,K00,LWX14} and references therein.

In this paper, we investigate several rigidity problems for hypersurfaces with constant shifted curvature functions embedded in the warped product manifolds.

For $n \geq 2$, let $\left(N^{n}, g_N\right)$ be an $n$-dimensional compact Riemannian manifold. We consider the warped product manifold $M^{n+1}=[0, \bar{r}) \times N^n$ $(0 < \bar{r} \leq \infty)$ equipped with the metric
$$
\bar{g}=d r^2+\lambda(r)^2 g_N,
$$
where $\lambda:[0, \bar{r})\rightarrow\mathbb{R}$ is a smooth positive function. Let $\Phi(r)$ be a function satisfying $\Phi'(r)=\lambda(r)$. Given $\varepsilon \in \mathbb{R}$, let $\Sigma$ be a smooth hypersurface in $M^{n+1}$, its shifted principal curvatures are defined by ${\kappa} - \varepsilon=(\kappa_1 - \varepsilon, \cdots, \kappa_n - \varepsilon)$, which are eigenvalues of the shifted Weingarten matrix $(h_{i}^{j} - \varepsilon \delta_{i}^{j})$. Then a hypersurface $\Sigma$ is called \textit{strictly shifted $k$-convex} if $(\kappa-\varepsilon) \in$ ${\Gamma}_k^{+}$ for all $n$-tuples of shifted principal curvatures along $\Sigma$. Here ${\Gamma}_k^{+}$ is so-called G\aa{}rding cone, see \eqref{garding cone} for its definition. Denote $\nu$ as its unit outward normal vector and $u=\bar{g}(\lambda\partial_{r},\nu)$ is the support function of $\Sigma$. If $u>0$ everywhere on $\Sigma$, then we say that $\Sigma$ is \textit{star-shaped}.

Recently, the study of shifted principal curvatures has attracted a lot of attentions, particularly in $\mathbb{H}^{n+1}$. For overcoming the obstructions produced by negative curvature, these modifications are essential. For example, Andrews \cite{B94} and Brendle-Huisken \cite{BH17} considered fully nonlinear contracting curvature flows with speeds depending on shifted principal curvatures. In $\mathbb{H}^{n+1}$, the curvature flows and geometric inequalities involving positivity or non-negativity of $\kappa_{i}-1$ for $1\le i \le n$ have been studied extensively in \cite{BCW21,HLW22,LX,WWZ23}, etc.

In \cite{HWZ23}, Hu, Wei and Zhou established a new Heintze-Karcher type inequality involving the shifted mean curvature in $\mathbb{H}^{n+1}$. As applications, they obtained an Alexandrov-type theorem. Later, we extended this result to some weighted higher order mean curvatures and their linear combinations in \cite{SWW24}. The main purpose of this paper is regarding these rigidity results in the warped product manifolds.

\begin{theorem}\label{warped product thm1}
	Suppose $M^{n+1}=\left([0, \bar{r}) \times N^n (K), \bar{g}=d r^2+\lambda(r)^2 g_N\right)$ is a warped product manifold satisfying
	$$\lambda'(r)^2 - \lambda^{\prime \prime}(r) \lambda(r) < K, \quad \forall r \in(0, \bar{r}),$$
	where $\left(N^{n}(K), g_N\right)$ is a compact manifold with constant sectional curvature $K$. Let $\varepsilon$ be a real number and $\Sigma$ be a closed, strictly shifted $k$-convex star-shaped hypersurface in $M^{n+1}$. Assume that $\Sigma$ satisfies one of the following equations:
	\begin{itemize}
		\item[(i)] $2\leq l \leq k\leq n$, $\lambda'-\varepsilon u>0$ on $\Sigma$ and there are nonnegative $C^1$-smooth functions $\{a_i\}_{i=1}^{l-1}$ and $\{b_j\}_{j=l}^{k}$ defined on $\mathbb{R}^2$, which satisfy $\partial_{1}a_i, \partial_{2}b_j \leq 0$ and $\partial_{2}a_i, \partial_{1}b_j \geq 0$, at least one of them are positive, such that
		\begin{equation*}
			\sum_{i=1}^{l-1}a_i (\Phi(r),\varepsilon\Phi(r)-u) H_i ({\kappa}-\varepsilon)=\sum_{j=l}^{k}b_j (\Phi(r),\varepsilon\Phi(r)-u)H_j ({\kappa}-\varepsilon);
		\end{equation*}
		\item[(ii)] $1\leq l \leq k\leq n-1$ and there are nonnegative $C^1$-smooth functions $\{a_i\}_{i=0}^{l-1}$ and $\{b_j\}_{j=l}^{k}$ defined on $\mathbb{R}^2$, which satisfy $\partial_{1}a_i, \partial_{2}b_j \leq 0$ and $\partial_{2}a_i, \partial_{1}b_j \geq 0$, at least one of them are positive, such that
		\begin{equation}\label{linear form}
			\sum_{i=0}^{l-1}a_i(\Phi(r),\varepsilon\Phi(r)-u) H_i ({\kappa}-\varepsilon)=\sum_{j=l}^{k}b_j (\Phi(r),\varepsilon\Phi(r)-u)H_j ({\kappa}-\varepsilon).
		\end{equation}
	\end{itemize}
	If $\partial_{2}a_i$ or $\partial_{2}b_j $ does not vanish at some point for some $0\le i\le l-1$ or $l\le j\le k$,  we further assume that $h_{ij}>\varepsilon g_{ij}$ on $\Sigma$. Then $\Sigma$ is a slice $\{r_0\} \times N(K)$ for some constant $r_0 $.
\end{theorem}

For any $\alpha>0$ and $\varepsilon\geq0$, if we set $b_{k}(s,t)=(\varepsilon s - t)^{\frac{1}{\alpha}}$, $a_{0}(s,t)=1$ and let all other terms vanish in (\ref{linear form}), we obtain the following result:

\begin{corollary}\label{coro}
	Suppose $M^{n+1}=\left([0, \bar{r}) \times N^n (K), \bar{g}=d r^2+\lambda(r)^2 g_N\right)$ is a warped product manifold satisfying
	$$\lambda'(r)^2 - \lambda^{\prime \prime}(r) \lambda(r) < K, \quad \forall r \in(0, \bar{r}),$$
	where $\left(N^{n}(K), g_N\right)$ is a compact manifold with constant sectional curvature $K$. Let $\varepsilon\geq0$ be a real number and $\Sigma$ be a closed hypersurface in $M^{n+1}$ with $h_{ij}>\varepsilon g_{ij}$ on $\Sigma$. If $\Sigma$ satisfies
	\begin{equation}\label{self similar soulution}
		\left(H_{k}(\kappa- \varepsilon)\right)^{-\alpha} = {u},
	\end{equation}
	for any fixed $k$ with $1\leq k\leq n-1$ and $\alpha>0$, then $\Sigma$ is a slice $\{r_0\} \times N(K)$ for some constant $r_0 $.
\end{corollary}

\begin{remark}
	Hypersurfaces satisfying (\ref{self similar soulution}) can be seen as self-similar solutions to the shifted curvature flows expanding by $\left(H_{k}(\kappa- \varepsilon)\right)^{-\alpha}$. When $\varepsilon=0$, Gao \cite[Corollary 4]{G23} derived this result under weaker convexity assumption and held for immersed hypersurfaces.
\end{remark}

Theorem \ref{warped product thm1} extends \cite[Theorem 2]{KLP18} and \cite[Theorem 2 and Theorem 3]{WX14}. For the same rigidity problems in the space forms, the star-shapedness is not necessary.

\begin{theorem}\label{space forem thm1}
	Let $\varepsilon$ be a real number and $\Sigma$ be a closed strictly shifted $k$-convex hypersurface in space form $M^{n+1}=\mathbb{R}^{n+1}$, $\mathbb{H}^{n+1}$ or $\mathbb{S}_{+}^{n+1}$. Assume that $\Sigma$ satisfies one of the following equations:
	\begin{itemize}
		\item[(i)] $2\leq l \leq k\leq n$, $\lambda'-\varepsilon u>0$ on $\Sigma$ and there are nonnegative $C^1$-smooth functions $\{a_i\}_{i=1}^{l-1}$ and $\{b_j\}_{j=l}^{k}$ defined on $\mathbb{R}^2$, which satisfy $\partial_{1}a_i, \partial_{2}b_j \leq 0$ and $\partial_{2}a_i, \partial_{1}b_j \geq 0$, at least one of them are positive, such that
		\begin{equation*}
			\sum_{i=1}^{l-1}a_i (\Phi(r),\varepsilon\Phi(r)-u) H_i ({\kappa}-\varepsilon)=\sum_{j=l}^{k}b_j (\Phi(r),\varepsilon\Phi(r)-u)H_j ({\kappa}-\varepsilon);
		\end{equation*}
		\item[(ii)] $\Sigma$ is star-shaped, $1\leq l \leq k\leq n-1$ and there are nonnegative $C^1$-smooth functions $\{a_i\}_{i=0}^{l-1}$ and $\{b_j\}_{j=l}^{k}$ defined on $\mathbb{R}^2$, which satisfy $\partial_{1}a_i, \partial_{2}b_j \leq 0$ and $\partial_{2}a_i, \partial_{1}b_j \geq 0$, at least one of them are positive, such that
		\begin{equation*}
			\sum_{i=0}^{l-1}a_i (\Phi(r),\varepsilon\Phi(r)-u) H_i ({\kappa}-\varepsilon)=\sum_{j=l}^{k}b_j (\Phi(r),\varepsilon\Phi(r)-u)H_j ({\kappa}-\varepsilon).
		\end{equation*}
	\end{itemize}
	If $\partial_{2}a_i$ or $\partial_{2}b_j $ does not vanish at some point for some $0\le i\le l-1$ or $l\le j\le k$,  we further assume that $h_{ij}>\varepsilon g_{ij}$ on $\Sigma$. Then $\Sigma$ is a geodesic sphere.
\end{theorem}

For $\varepsilon=1$ and $M^{n+1}=\mathbb{H}^{n+1}$, it reduces to \cite[Theorem 1.2 and Theorem 4.1]{SWW24}. As a direct consequence of Theorem \ref{space forem thm1}(ii), we have the following

\begin{corollary}\label{cor1}
	Let $1\leq k< n$ be an integer and $\Sigma$ be a closed, star-shaped strictly shifted $k$-convex hypersurface in space form $M^{n+1}=\mathbb{R}^{n+1}$, $\mathbb{H}^{n+1}$ or $\mathbb{S}_{+}^{n+1}$. Assume that $\varepsilon\in\mathbb{R}$ is a constant. If there is positive $C^1$-smooth function $\chi$ defined on $\mathbb{R}^2$, which satisfies $\partial_{1}\chi \leq 0$ and $\partial_{2}\chi  \geq 0$, such that
	\begin{equation}
		{H_{k}({\kappa}-\varepsilon)}=\chi(\Phi(r),\varepsilon\Phi(r)-u).
	\end{equation}
	If $\partial_{2}\chi$  does not vanish at some point, we further assume that $h_{ij}>\varepsilon g_{ij}$ on $\Sigma$. Then $\Sigma$ is a geodesic sphere.
\end{corollary}

\begin{remark}
	Corollary \ref{cor1} illustrates that the conditions of static-convex and $\lambda'-\varepsilon u>0$ are unnecessary in \cite[Theorem 1.6]{LWX25} under the star-shapedness assumption. When $M^{n+1}=\mathbb{H}^{n+1}$, $\varepsilon=0$ and $\chi=1/\cosh(r)$, it was proved in \cite[Theorem 1.2]{W16}.
\end{remark}

In \cite{S60}, Stong established some rigidity results under the condition $H_{k-1}/H_{k}\geq c \geq H_{k-2}/H_{k-1}$ with constant $c$. As an application of Theorem \ref{space forem thm1}, we prove the following results for shifted higher order mean curvatures and generalize the constant $c$ to a function.

\begin{corollary}\label{coro3}
	Let $1<k\leq n$ be an integer and $\Sigma$ be a closed, star-shaped strictly shifted $k$-convex hypersurface in space form $M^{n+1}=\mathbb{R}^{n+1}$, $\mathbb{H}^{n+1}$ or $\mathbb{S}_{+}^{n+1}$. Assume that $\lambda'-\varepsilon u>0$ on $\Sigma$, where $\varepsilon\in\mathbb{R}$ is a constant. If there is a positive $C^1$-smooth function $\chi$ defined on $\mathbb{R}^2$, which satisfies $\partial_{1}\chi \geq 0$ and $\partial_{2}\chi  \leq 0$, such that
	\begin{equation}\label{pluscondition1}
		\frac{H_{k-1}({\kappa}-\varepsilon)}{H_{k}({\kappa}-\varepsilon)}\geq \chi(\Phi(r),\varepsilon\Phi(r)-u)\geq\frac{H_{k-2}({\kappa}-\varepsilon)}{H_{k-1}({\kappa}-\varepsilon)}.
	\end{equation}
	If $\partial_{2}\chi $ does not vanish at some point, we further assume that $h_{ij}>\varepsilon g_{ij}$ on $\Sigma$. Then $\Sigma$ is a geodesic sphere.
\end{corollary}

Next, we investigate some rigidity results of non-linear form for static-convex domains in certain sub-static warped product manifolds (see Section \ref{sec2}). Here, we recall a closed smooth hypersurface $\Sigma \subset M^{n+1}$ is called \textit{static-convex} if its second fundamental form satisfies
\begin{equation}\label{static convex}
	h_{ij}\geq\frac{\bar{\nabla}_{\nu} \lambda'}{\lambda'}g_{ij}, \quad \text{everywhere on }\Sigma.
\end{equation}
Applying the Minkowski type inequality by Li and Xia \cite{LX19} for static-convex domains in sub-static warped product manifolds, Li, Wei and Xu \cite{LWX25} proved Heintze-Karcher type inequality involving a general shifted factor $\varepsilon\in \mathbb{R}$ (see Proposition \ref{space form hkp}). Combining the Heintze-Karcher type inequality and Minkowski type formulas, we prove the following theorem.

\begin{theorem}\label{warped product thm+2}
	Suppose $M^{n+1}=\left([0, \bar{r}) \times N^n (K), \bar{g}=d r^2+\lambda(r)^2 g_N\right)$ is a sub-static warped product manifold satisfying
	\begin{equation}\label{condition wp}
		\lambda'(r)^2 - \lambda^{\prime \prime}(r) \lambda(r) < K, \quad \forall r \in(0, \bar{r}),
	\end{equation}
	where $\left(N^{n}(K), g_N\right)$ is a compact manifold with constant sectional curvature $K$. Let $\Omega$ be a bounded domain with star-shaped, strictly shifted $k$-convex and static-convex boundary $\Sigma=\partial\Omega$ in $M^{n+1}$. Assume that $\lambda'>0$ and $\lambda'-\varepsilon u>0$ on $\Sigma$, where $\varepsilon\in\mathbb{R}$ is a constant. If there are nonnegative $C^1$-smooth functions $\{a_j\}_{j=1}^{k}$ and $\{b_j\}_{j=1}^{k}$ defined on $\mathbb{R}^2$, which satisfy $\partial_{1}a_j, \partial_{1}b_j \geq 0$ and $\partial_{2}a_j, \partial_{2}b_j \leq 0$, such that
	\begin{equation}\label{bjhj*}
		\begin{aligned}
			& \sum_{j=1}^{k}\bigg(a_j (\Phi(r),\varepsilon\Phi(r)-u) H_j ({\kappa}-\varepsilon) + b_j (\Phi(r),\varepsilon\Phi(r)-u) H_1 ({\kappa}-\varepsilon) H_{j-1} ({\kappa}-\varepsilon)\bigg) \\
			& = \eta (\Phi(r),\varepsilon\Phi(r)-u),
		\end{aligned}
	\end{equation}
	for some $C^1$-smooth positive function $\eta$ which satisfies $\partial_{1}\eta \leq 0$ and $\partial_{2}\eta \geq 0$. If $\partial_{2}a_j, \partial_{2}b_j$ or $\partial_{2}\eta$ does not vanish at some point for some $1\le j\le k$, we further assume that $h_{ij}>\varepsilon g_{ij}$ on $\Sigma$. Then $\Sigma$ is a slice $\{r_0\} \times N(K)$ for some constant $r_0 $.
\end{theorem}

When $\varepsilon=0$, Theorem \ref{warped product thm+2} reduces to \cite[Theorem 1]{KLP18}. The same result also applies to the space forms. For this case, if $b_j =0$ for all $j$, $1\le j\le k$ in (\ref{bjhj*}), then the star-shapedness assumption is unnecessary. When $\varepsilon=1$ and $M^{n+1}=\mathbb{H}^{n+1}$, Theorem \ref{warped product thm+2} reduces to \cite[Theorem 1.4]{SWW24}. When $\varepsilon=-1$ and $M^{n+1}=\mathbb{H}^{n+1}$, by using the unit normal flow in hyperbolic space $\mathbb{H}^{n+1}$, Li, Wei and Xu \cite{LWX25} proved a Heintze-Karcher type inequality for non mean-convex domains (see Proposition \ref{hkp}). As a application, we obtain the following similar result for non mean-convex domains in $\mathbb{H}^{n+1}$.

\begin{theorem}\label{thm+2}
	 Let $\Sigma$ be a closed star-shaped hypersurface in $\mathbb{H}^{n+1}$ with $H_k ({\kappa}+1)>0$. If there are nonnegative $C^1$-smooth functions $\{a_j\}_{j=1}^{k}$ and $\{b_j\}_{j=1}^{k}$ defined on $[1,\infty)\times(-\infty,0)$, which satisfy $\partial_{1}a_j, \partial_{1}b_j \geq 0$ and $\partial_{2}a_j, \partial_{2}b_j \leq 0$, such that
	\begin{equation}\label{bjhj}
		\begin{aligned}
			& \sum_{j=1}^{k}\bigg(a_j (\lambda',-\lambda'-u) H_j ({\kappa}+1) + b_j (\lambda',-\lambda'-u) H_1 ({\kappa}+1) H_{j-1} ({\kappa}+1)\bigg) \\
			& = \eta (\lambda',-\lambda'-u),
		\end{aligned}
	\end{equation}
	for some $C^1$-smooth positive function $\eta$ which satisfies $\partial_{1}\eta \leq 0$ and $\partial_{2}\eta \geq 0$. If $\partial_{2}a_j, \partial_{2}b_j$ or $\partial_{2}\eta$ does not vanish at some point for some $1\le j\le k$, we further assume that $h_{ij}>- g_{ij}$ on $\Sigma$. Then $\Sigma$ is a geodesic sphere.
\end{theorem}

Notice that $H_k ({\kappa}+1)$ can be expressed as some linear combination of $H_i ({\kappa})$, i.e.
\begin{equation*}
	H_{k}({\kappa}+1) = \sum_{i=0}^{k}\binom{k}{i}H_{i}(\kappa).
\end{equation*}
Therefore, as a direct consequence of Theorem \ref{thm+2}, we have
\begin{corollary}\label{coro1}
Let $\Sigma$ be a closed hypersurface in $\mathbb{H}^{n+1}$ with $H_k ({\kappa}+1)>0$. If there are nonnegative constants $a_0$ and $\{a_j\}_{j=1}^{k}$, at least one of them not vanishing, such that
\begin{equation*}
a_0 - \sum_{j=1}^{k}a_{j}=\sum_{j=1}^{k} \sum_{l=0}^{k-j}\binom{l+j}{j}a_{l+j} H_j (\kappa),
\end{equation*}
then $\Sigma$ is a geodesic sphere.
\end{corollary}

\begin{remark}
	If $a_0 > \sum_{j=1}^{k}a_{j}$ and $H_k ({\kappa})>0$, then the conclusion follows from \cite[Theorem 11]{WX14}. Here, we only need $H_k ({\kappa}+1)>0$.
\end{remark}

In \cite[Theorem 3]{KLP18}, Kwong, Lee and Pyo proved a rigidity theorem for self-expanding solitons to the weighted generalized inverse curvature flow in $\R^{n+1}$
\begin{equation}
	\frac{d}{dt}X=\sum_{0 \leq i < j \leq n} a_{i,j}\left(\frac{H_{i}}{H_{j}}\right)^{\frac{1}{j-i}}\nu,
\end{equation}
where the weight functions $\left\{a_{i,j}(x)|0 \leq i < j \leq n\right\}$ are non-negative functions on the hypersurface satisfying $\sum_{0 \leq i < j \leq n}a_{i,j}(x)=1$. We next extend it to the warped product manifolds.

\begin{theorem}\label{thm+3}
Suppose $M^{n+1}=\left([0, \bar{r}) \times N^n (K), \bar{g}=d r^2+\lambda(r)^2 g_N\right)$ is a warped product manifold satisfying
$$\lambda'(r)^2 - \lambda^{\prime \prime}(r) \lambda(r) < K, \quad \forall r \in(0, \bar{r}),$$
where $\left(N^{n}(K), g_N\right)$ is a compact manifold with constant sectional curvature $K$. Let $\Sigma$ be a closed strictly shifted $k$-convex hypersurface in $M^{n+1}$ and $\left\{a_{i,j}(x)|0 \leq i < j \leq n\right\}$ be non-negative functions on $\Sigma$ satisfying $\sum_{0 \leq i < j \leq n}a_{i,j}(x)=1$, where $k=\max\left\{j|a_{i,j}>0\right.$ for some $\left.0 \leq i < j \leq n\right\}$. Assume that $\lambda'-\varepsilon u > 0$ on $\Sigma$, where $\varepsilon \in \mathbb{R}$ is a constant. If there exists a constant $\beta > 0$ satisfying
\begin{equation}\label{aij}
\sum_{0 \leq i < j \leq k} a_{i,j}\left(\frac{H_{i}({\kappa}-\varepsilon)}{H_{j}({\kappa}-\varepsilon))}\right)^{\frac{1}{j-i}} = \beta \frac{u}{\lambda'- \varepsilon u},
\end{equation}
then $\Sigma$ is a slice $\{r_0\} \times N(K)$ for some constant $r_0 $.

\end{theorem}

We also prove the following rigidity result for nonhomogeneous curvature functions.
\begin{theorem}\label{thm+4}
	Suppose $M^{n+1}=\left([0, \bar{r}) \times N^n (K), \bar{g}=d r^2+\lambda(r)^2 g_N\right)$ is a sub-static warped product manifold satisfying
	$$\lambda'(r)^2 - \lambda^{\prime \prime}(r) \lambda(r) < K, \quad \forall r \in(0, \bar{r}),$$
	where $\left(N^{n}(K), g_N\right)$ is a compact manifold with constant sectional curvature $K$. Let $\Omega$ be a bounded domain with smooth, strictly shifted $k$-convex and static-convex boundary $\Sigma=\partial\Omega$ in $M^{n+1}$. Assume that $\lambda'>0$ and $\lambda'-\varepsilon u>0$ on $\Sigma$, where $\varepsilon\in\mathbb{R}$ is a constant. If $\Sigma$ satisfies
	\begin{equation}\label{pluscondition2}
		\left(H_{k}(\kappa- \varepsilon)\right)^{-\alpha} = \frac{u}{\lambda'- \varepsilon u},
	\end{equation}
	for any fixed $k$ with $1\leq k\leq n$ and $\alpha\geq\frac{1}{k}$, then $\Sigma$ is a slice $\{r_0\} \times N(K)$ for some constant $r_0 $.
\end{theorem}

\begin{remark}
	We notice that hypersurfaces satisfying (\ref{pluscondition2}) are self-similar solutions to the curvature flow expanding by $H_k^{-\alpha}$ when $M^{n+1}=\mathbb{R}^{n+1}$ and $\varepsilon=0$. For static-convex $\Sigma$, Theorem \ref{thm+4} reduces to \cite[Theorem 3]{GM21} when $\varepsilon=0$. According to \cite[Theorem 1.1]{HWZ23} and Proposition \ref{hkp}, the static-convex assumption of $\Sigma$ is unnecessary when $\varepsilon=1,-1$ and $M^{n+1}=\mathbb{H}^{n+1}$.
\end{remark}

Finally, we consider the rigidity problems in a class of warped product spaces $M^{n+1}$, where $\left(N^{n}, g_N\right)$ does not have constant sectional curvature.

\begin{theorem}\label{glwthm}
	Suppose $M^{n+1}=\left([0, \bar{r}) \times N^n, \bar{g}=d r^2+\lambda(r)^2 g_N\right)$ is a sub-static warped product manifold satisfying
	\begin{equation}\label{glwcondition}
		\begin{aligned}
			&\operatorname{Ric}_N \geq (n-1) K g_N,\\
			&\lambda'(r)^2 - \lambda^{\prime \prime}(r) \lambda(r) < K, \quad \forall r \in(0, \bar{r}),
		\end{aligned}
	\end{equation}
	where $\left(N^{n}, g_N\right)$ is a closed manifold, $K > 0$ is a constant and $\operatorname{Ric}_N$ is the Ricci curvature of $g_N$. Let $\Omega$ be a bounded domain with static-convex boundary $\Sigma=\partial\Omega$ in $M^{n+1}$. Assume that $\lambda' > 0$ and $\lambda'-\varepsilon u>0$ on $\Sigma$, where $\varepsilon\in\mathbb{R}$ is a constant. If there are positive $C^1$-smooth functions $a$ and $b$ defined on $\mathbb{R}^2$, which satisfy $\partial_{1}a, \partial_{1}b \geq 0$ and $\partial_{2}a, \partial_{2}b \leq 0$, and $\Sigma$ satisfies one of the following conditions:
	\begin{itemize}
		\item[(i)] $\Sigma$ is strictly shifted mean convex and
		$
			a (\Phi(r),\varepsilon\Phi(r)-u)H_1 ({\kappa}-\varepsilon)
		$
		is constant;
		\item[(ii)] $\Sigma$ is star-shaped, strictly shifted $2$-convex and
		$
			a (\Phi(r),\varepsilon\Phi(r)-u)H_2 ({\kappa}-\varepsilon)
		$
		is constant;
		\item[(iii)] $\Sigma$ is star-shaped, strictly shifted $2$-convex and
		$
			a (\Phi(r),\varepsilon\Phi(r)-u)\frac{H_2 ({\kappa}-\varepsilon)}{H_1 ({\kappa}-\varepsilon)}
		$
		is constant;
		\item[(iv)] $\Sigma$ is star-shaped, strictly shifted $2$-convex and
		\begin{equation}\label{glw condition}
			\begin{aligned}
				&a (\Phi(r),\varepsilon\Phi(r)-u) H_2 ({\kappa}-\varepsilon) + b (\Phi(r),\varepsilon\Phi(r)-u) (H_1 ({\kappa}-\varepsilon))^2 \\
				&= \eta (\Phi(r),\varepsilon\Phi(r)-u),
			\end{aligned}
		\end{equation}
		for some $C^1$-smooth positive function $\eta$ which satisfies $\partial_{1}\eta \leq 0$ and $\partial_{2}\eta \geq 0$.
	\end{itemize}
	If $\partial_{2}a, \partial_{2}b$ or $ \partial_{2}\eta$ dose not vanish at some point, we further assume that $h_{ij}>\varepsilon g_{ij}$ on $\Sigma$. Then $\Sigma$ is a slice $\{r_0\} \times N$ for some constant $r_0 $.
\end{theorem}

For the case of $\varepsilon=0$ in Theorem \ref{glwthm}, the convexity condition is superfluous since it is implied by the constancy of $H_k$.

\begin{theorem}\label{glwthm1}
	Suppose $M^{n+1}=\left([0, \bar{r}) \times N^n, \bar{g}=d r^2+\lambda(r)^2 g_N\right)$ is a sub-static warped product manifold satisfying
	\begin{equation}\label{glwcondition1}
		\begin{aligned}
			&\operatorname{Ric}_N \geq (n-1) K g_N,\\
			&\lambda'(r)^2 - \lambda^{\prime \prime}(r) \lambda(r) < K, \quad \forall r \in(0, \bar{r}),
		\end{aligned}
	\end{equation}
	where $\left(N^{n}, g_N\right)$ is a closed manifold, $K > 0$ is a constant and $\operatorname{Ric}_N$ is the Ricci curvature of $g_N$. Let $\Omega$ be a bounded domain with smooth boundary $\Sigma=\partial\Omega$ in $M^{n+1}$. Assume that $\lambda'(r) > 0$ for all $r\in(0,\bar{r})$, $\lambda'(0)=0$ and $\lambda''(0)>0$. If there exists positive $C^1$-smooth functions $a$ defined on $\mathbb{R}^2$, which satisfies $\partial_{1}a \geq 0$ and $\partial_{2}a \leq 0$, and $\Sigma$ satisfies one of the following conditions:
	\begin{itemize}
		\item[(i)]
		$
			a (\Phi(r),-u)H_1 ({\kappa})
		$
		is constant;
		\item[(ii)] $\Sigma$ is star-shaped and
		$
			a (\Phi(r),-u)H_2 ({\kappa})
		$
		is constant;
		\item[(iii)] $\Sigma$ is star-shaped with positive mean curvature and
		$
			a (\Phi(r),-u)\frac{H_2 ({\kappa})}{H_1 ({\kappa})}
		$
		is constant.
	\end{itemize}
	If $\partial_{2}a$ dose not vanish at some point, we further assume that $\Sigma$ is strictly convex. Then $\Sigma$ is a slice $\{r_0\} \times N$ for some constant $r_0 $.
\end{theorem}

\begin{remark}
	If $a=1$, case (i) of Theorem \ref{glwthm1} reduces to \cite[Theorem 1.1]{B13}. If $a = (\lambda')^\alpha$, Theorem \ref{glwthm1} says that it is not necessary to require that $\left(N^{n}, g_N\right)$ has constant sectional curvature for the case of $k=2$ in \cite[Theorem 1.3]{LWX14}.
	
\end{remark}

The paper is organized as follows: Section \ref{sec2} reviews some basic facts on the higher order mean curvatures. We provide some Minkowski type formulas and inequalities, which are the key tools of this paper. Sections \ref{sec3} and \ref{sec4} present proofs of the main results, Theorems \ref{warped product thm1}--\ref{thm+2}. Section \ref{sec5} focuses on the results of non-linear curvature conditions, Theorems \ref{thm+3} and \ref{thm+4}. In Section \ref{sec6}, we consider the rigidity results in more general warped product manifolds where the fiber does not have constant sectional curvature and show Theorems \ref{glwthm} and \ref{glwthm1}.

\section{Preliminaries}\label{sec2}
In this section, we begin by recalling some fundamental definitions and properties of elementary symmetric functions.

Let $\sigma_{k}$ be the $k$th elementary symmetry function $\sigma_{k}:\mathbb{R}^{n}\rightarrow\mathbb{R}$ defined by
$$
\sigma_k(\lambda)=\sum_{i_1<\cdots<i_k} \lambda_{i_1} \cdots \lambda_{i_k} \text { for } \lambda=\left(\lambda_1, \cdots, \lambda_{n}\right) \in \mathbb{R}^{n}.
$$
For a symmetric $n \times n$ matrix $A=\left(A_i^j\right)$, we set
\begin{equation}\label{sgmk}
	\sigma_k(A)=\frac{1}{k !} \delta_{j_1 \cdots j_k}^{i_1 \cdots i_k} A_{i_1}^{j_1} \cdots A_{i_k}^{j_k},
\end{equation}
where $\delta_{j_1 \cdots j_k}^{i_1 \cdots i_k}$ is the generalized Kronecker symbol.
If $\lambda(A)=\left(\lambda_1(A), \cdots, \lambda_n(A)\right)$ are the real eigenvalues of $A$, then
$$
\sigma_k(A)=\sigma_k(\lambda(A)) .
$$
We further introduce the notation
$$
\sigma_{k;j}(A):=\sigma_{k}(\lambda_1(A),\cdots,\lambda_{j-1}(A),\lambda_{j+1}(A),\cdots\lambda_{n}(A)), \quad \forall 1\leq k \leq n-1.
$$
The $k$th Newton transformation is defined as
$$
\left(T_k\right)_i^j(A)=\frac{\partial \sigma_{k+1}}{\partial A_j^i}(A)=\frac{1}{k !}\delta_{i i_1 \cdots i_{k}}^{j j_1 \cdots j_{k}} A_{j_1}^{i_1} \cdots A_{j_{k}}^{i_{k}}.
$$
For a diagonal matrix $A$ with eigenvalues $\lambda(A)=\left(\lambda_1(A), \cdots, \lambda_n(A)\right)$, this simplifies to
$$
\left(T_k\right)_i^j(A) = \frac{\partial \sigma_{k+1}}{\partial \lambda_i}(\lambda(A))\delta_{i}^{j}.
$$
The following formulas for the elementary symmetric functions are well-known.
\begin{lemma}[\cite{R74}]\label{tklemma}
	We have
	
	\begin{equation}\label{tk1}
		\sum_{i, j}\left(T_{k-1}\right)_{i}^{j}(A) A_{j}^{i}=k \sigma_{k}(A) .
	\end{equation}
	\begin{equation}\label{tk2}
		\sum_{i, j}\left(T_{k-1}\right)_{i}^{j}(A) \delta_{j}^{i}=(n+1-k) \sigma_{k-1}(A) .
	\end{equation}
\end{lemma}
For each $k=1,\cdots,n$, let
\begin{equation}\label{garding cone}
	\Gamma_{k}^{+}=\left\{\lambda \in \mathbb{R}^n: \sigma_i(\lambda)>0,1 \leq i \leq k\right\}
\end{equation}
be the G\r{a}rding cone. A symmetric matrix $A$ is said to belong to $\Gamma_k^{+}$ if its eigenvalues $\lambda(A) \in \Gamma_k^{+}$. We denote by
\begin{equation}
	H_k (\lambda)= \frac{1}{C_{n}^k}\sigma_k (\lambda)
\end{equation}
the normalized $k$th elementary symmetry function. As a convention, we take $H_0 = 1$, $H_{-1} = 0$. The following classical Newton-Maclaurin inequalities will be useful.
\begin{lemma}[\cite{G02}]
	 For $1 \leq l<k \leq n$ and $\lambda \in {\Gamma_{k}^{+}}$, the following inequalities hold:
	\begin{equation}\label{nm1}
		H_{k-1}(\lambda)H_{l}(\lambda) \geq H_{k}(\lambda)H_{l-1}(\lambda).
	\end{equation}
	\begin{equation}\label{nm2}
		H_{l}(\lambda) \geq H_{k}(\lambda)^{\frac{l}{k}}.
	\end{equation}
	Moreover, equality holds in (\ref{nm1}) or (\ref{nm2}) at $\lambda$ if and only if $\lambda=c(1,1, \cdots, 1)$.
\end{lemma}

Next, we collect results for the warped product manifolds. Let $\left(N^{n}, g_N\right)$ be an $n$-dimensional compact Riemannian manifold. We consider the warped product manifold $M^{n+1}=[0, \bar{r}) \times N^n$ $(0 < \bar{r} \leq \infty)$ equipped with the metric
$$
\bar{g}=d r^2+\lambda(r)^2 g_N,
$$
where $\lambda:[0, \bar{r})\rightarrow\mathbb{R}$ is a smooth positive function. A Riemannian manifold $(\mathcal{M},g)$ which admits a nontrivial smooth solution $f$ to $\Delta_{\mathcal{M}}fg-\nabla^{2}_{\mathcal{M}}f+f\operatorname{Ric}_{\mathcal{M}}=0$ is called a {static} manifold and to $\Delta_{\mathcal{M}}fg-\nabla^{2}_{\mathcal{M}}f+f\operatorname{Ric}_{\mathcal{M}}\geq0$ is called a {sub-static} manifold, where $f$ is called the potential function. This name comes from general relativity since such manifolds are related to the concept of static spacetimes. As demonstrated in \cite[Proposition 2.1]{B13}, $M^{n+1}$ is a sub-static warped product manifold with potential $\lambda'(r)$ if and only if the following inequality holds for some constant $K\in\mathbb{R}$:
\begin{equation*}
	\begin{aligned}
	     \lambda' \left(\operatorname{Ric}_N - (n-1) K g_N\right) + \frac{1}{2} \lambda^3 \frac{d}{dr} \left(2 \frac{\lambda''}{\lambda} - (n-1) \frac{K-{\lambda'}^2}{\lambda^2}\right) \geq 0.
	\end{aligned}
\end{equation*}
Examples of sub-static manifolds include the space forms, Schwarzschild and Reissner-Nordstrom manifolds, see e.g. (\cite{B13,BHW16}).

Let $\Sigma$ be a smooth hypersurface in $M^{n+1}$. Denote $\bar{\nabla}$ and $\nabla$ as the Levi-Civita connection on $M^{n+1}$ and $\Sigma$, respectively. Let $\{e_i\}_{i=1}^{n}$ and $\nu$ be an orthonormal basis and the unit outward normal of $\Sigma$, respectively. Then the second fundamental form $h=\left(h_{i j}\right)$ of $\Sigma$ in $M^{n+1}$ is given by
$
h_{i j}=\left\langle\bar{\nabla}_{e_i} \nu, e_j\right\rangle .
$
The principal curvatures $\kappa=\left(\kappa_1, \cdots, \kappa_n\right)$ of $\Sigma$ are the eigenvalues of the Weingarten matrix $\left(h_i^j\right)=\left(g^{j k} h_{k i}\right)$, where $\left(g^{i j}\right)$ is the inverse matrix of $\left(g_{i j}\right)$. We define
$$
\Phi(r)=\int_{0}^{r}\lambda(s)ds.
$$
The vector field $V=\bar{\nabla}\Phi=\lambda\partial_{r}$ on $M^{n+1}$ is a conformal Killing field, i.e., $\bar{\nabla}(\lambda\partial_{r})=\lambda' \bar{g}$. We notice that $\overline{\operatorname{div}}(V)=(n+1)\lambda'$. A hypersurface $\Sigma$ in $M^{n+1}$ is called star-shaped if its support function
\begin{equation}\label{spf}
	u=\left\langle \lambda(r)\partial_{r}, \nu \right\rangle > 0.
\end{equation}

The following formulas hold for smooth hypersurfaces in $M^{n+1}$, see e.g.\cite{B13}.

\begin{lemma}
	Let $\Sigma$ be a smooth hypersurface in $M^{n+1}$. Then we have
	\begin{equation}\label{hession phi}
		\nabla_{i}\Phi=\langle V,e_{i}\rangle, \quad \nabla_{j}\nabla_{i}\Phi=\lambda'g_{ij}-uh_{ij},
	\end{equation}
	and the support function $u=\langle V,\nu\rangle$ satisfies
	\begin{equation}\label{hession u}
		\nabla_{i}u=\langle V,e_{k}\rangle h_{i}^{k}.
	\end{equation}
\end{lemma}

\begin{lemma}\label{wpmkl}
	Let $\Sigma$ be a closed hypersurface in the warped product manifolds $ M^{n+1}=\left([0, \bar{r}) \times N^n (K), \bar{g}=d r^2+\lambda(r)^2 g_N\right)$, where $\left(N^{n}(K), g_N\right)$ is a compact manifold with constant sectional curvature $K$. For any $\varepsilon\in\mathbb{R}$, we have
	\begin{equation}\label{h1}
		\int_{\Sigma} u H_1({\kappa}-\varepsilon) d \mu=\int_{\Sigma}\left( \lambda'- \varepsilon u \right) d \mu;
	\end{equation}
	 and for $2\leq k\leq n$,
	\begin{equation}\label{wpmkf}
		\begin{aligned}
			\int_{\Sigma} u H_k({\kappa}-\varepsilon) d \mu=&\int_{\Sigma}\left( \lambda'- \varepsilon u \right) H_{k-1}({\kappa}-\varepsilon) d \mu+(k-1)\int_{\Sigma}\sum_{j=1}^{n}A_{j}H_{k-2;j}(\tilde{h})d \mu,
		\end{aligned}
	\end{equation}
	where $\tilde{h}_j^i=h_j^i- \varepsilon\delta_j^i$, $A_{j}=-\frac{1}{n(n-1)}\operatorname{Ric}(e_j,\nu)\nabla_{j}\Phi$ and $\operatorname{Ric}$ is the Ricci curvature of $\bar{g}$. Moreover, if $\Sigma$ is star-shaped and $(M,\bar{g})$ satisfies
	$$\frac{\lambda^{\prime \prime}(r)}{\lambda(r)}+\frac{K-\lambda^{\prime}(r)^2}{\lambda(r)^2}>0, \quad \forall r \in(0, \bar{r}),$$
	then
	\begin{equation}\label{Aj}
		A_j \geq 0, \quad \forall 1 \leq j \leq n.
	\end{equation}
\end{lemma}
\begin{proof}
	It follows from (\ref{hession phi}) that
		\begin{align}
			\sum_{i,j=1}^{n}\nabla_{i}\left(\nabla_{j}\Phi \left(T_{k-1}\right)^{ij}(\tilde{h})\right)=&\lambda'\sum_{i,j=1}^{n}\left(T_{k-1}\right)^{ij}(\tilde{h})g_{ij}-u\sum_{i,j=1}^{n}\left(T_{k-1}\right)^{ij}(\tilde{h})h_{ij} \label{wp2} \\
			&+\sum_{i,j=1}^{n}\nabla_{j}\Phi\nabla_{i}\left(\left(T_{k-1}\right)^{ij}(\tilde{h})\right) \notag \\
			=&(\lambda'-\varepsilon u)\sum_{i,j=1}^{n}\left(T_{k-1}\right)^{ij}(\tilde{h})g_{ij}-u\sum_{i,j=1}^{n}\left(T_{k-1}\right)^{ij}(\tilde{h})(h_{ij}-\varepsilon g_{ij} ) \notag \\
			&+\sum_{i,j=1}^{n}\nabla_{j}\Phi\nabla_{i}\left(\left(T_{k-1}\right)^{ij}(\tilde{h})\right) \notag \\
			=&(\lambda'-\varepsilon u)(n+1-k)\sigma_{k-1}(\kappa-\varepsilon)-uk\sigma_{k}(\kappa-\varepsilon) \notag \\
			&+\sum_{i,j=1}^{n}\nabla_{j}\Phi\nabla_{i}\left(\left(T_{k-1}\right)^{ij}(\tilde{h})\right) \notag ,
		\end{align}
	where in the last equality we used (\ref{tk1}) and (\ref{tk2}).
	
	If $k=1$, we have
	\begin{equation*}
		\sum_{i,j=1}^{n}\nabla_{j}\Phi\nabla_{i}\left(\left(T_{k-1}\right)^{ij}(\tilde{h})\right) = 0.
	\end{equation*}
	Thus integrating (\ref{wp2}) and applying the divergence theorem, we get (\ref{h1}).
	
	For $2 \leq k \leq n$, taking an orthogonal basis $\{e_{1}, \cdots,e_{n}\}$ such that $g_{ij}=\delta_{ij}$ and $h_{ij}=\kappa_{i}\delta_{ij}$. As in the proof of \cite[Proposition 8]{BE13}, we can get
	\begin{equation}\label{wp1}
		\begin{aligned}
			\sum_{i,j=1}^{n}\nabla_{j}\Phi\nabla_{i}\left(\left(T_{k-1}\right)^{ij}(\tilde{h})\right)=&-\frac{n+1-k}{n-1}\sum_{i,j=1}^{n}\left(T_{k-2}\right)^{ij}(\tilde{h})\operatorname{Ric}(e_i,\nu)\nabla_{j}\Phi\\
			=&-\frac{n+1-k}{n-1}C_{n-1}^{k-2}\sum_{j=1}^{n}H_{k-2;j}(\tilde{h})\operatorname{Ric}(e_j,\nu)\nabla_{j}\Phi.
		\end{aligned}
	\end{equation}
	Putting (\ref{wp1}) into (\ref{wp2}) and integrating by parts, we obtain (\ref{wpmkf}).
	
	Moreover, the Ricci curvature of $(M,\bar{g})$ is given by \cite[(2)]{B13}
	\begin{equation}\label{Ric}
		\begin{aligned}
			\operatorname{Ric} = & - \left(\frac{\lambda^{\prime \prime}(r)}{\lambda(r)} - (n-1) \frac{K-\lambda^{\prime}(r)^2}{\lambda(r)^2}\right) \bar{g} \\
			& -(n-1) \left(\frac{\lambda^{\prime \prime}(r)}{\lambda(r)}+\frac{K-\lambda^{\prime}(r)^2}{\lambda(r)^2}\right) d r^2.
		\end{aligned}
	\end{equation}
	From (\ref{Ric}), we know that
	$$
	\operatorname{Ric}\left(e_j, \nu\right)=-(n-1)\left(\frac{\lambda^{\prime \prime}(r)}{\lambda(r)}+\frac{K-\lambda^{\prime}(r)^2}{\lambda(r)^2}\right) \frac{1}{\lambda(r)^2} u \nabla_j \Phi ,
	$$
	which implies
	$$
	\begin{aligned}
		A_j & =-\frac{1}{n(n-1)}\operatorname{Ric}(e_j,\nu)\nabla_{j}\Phi \\
		& =\frac{1}{n}\left(\frac{\lambda^{\prime \prime}(r)}{\lambda(r)}+\frac{K-\lambda^{\prime}(r)^2}{\lambda(r)^2}\right) \frac{1}{\lambda(r)^2} u |\nabla_j \Phi|^2.
	\end{aligned}
	$$
	Using the star-shapedness (\ref{spf}) of $\Sigma$ and the assumption $$\frac{\lambda^{\prime \prime}(r)}{\lambda(r)}+\frac{K-\lambda^{\prime}(r)^2}{\lambda(r)^2}>0, \quad \forall r \in(0, \bar{r}),$$
    we conclude $A_j \geq 0$ for any $j=1, \cdots, n$.
	
\end{proof}

In order to prove the rigidity result on weighted shifted curvature functions later, we need to extend the above lemma to the following type.

\begin{lemma}\label{wppml}
	Let $\phi$ be a $C^1$-smooth function on a closed hypersurface $\Sigma$ in the warped product manifolds $ M^{n+1}=\left([0, \bar{r}) \times N^n (K), \bar{g}=d r^2+\lambda(r)^2 g_N\right)$, where $\left(N^{n}(K), g_N\right)$ is a compact manifold with constant sectional curvature $K$. For any $\varepsilon\in\mathbb{R}$, we have
	\begin{equation}\label{wppmk}
		\begin{aligned}
			\int_{\Sigma} \phi u H_k({\kappa}-\varepsilon) d \mu= & \int_{\Sigma} \phi(\lambda'-\varepsilon u) H_{k-1}({\kappa}-\varepsilon) d \mu +(k-1)\int_{\Sigma}\sum_{j=1}^{n}A_{j}\phi H_{k-2;j}(\tilde{h})d \mu\\
		   & +  \frac{1}{k C_{n}^k} \int_{\Sigma} \left(T_{k-1}\right)^{ij}(\tilde{h})\nabla_{i}\phi \nabla_j \Phi d \mu,
		\end{aligned}
	\end{equation}
	where $\tilde{h}_j^i=h_j^i- \varepsilon\delta_j^i$.
	Moreover, if $\phi=\chi(\Phi(r),\varepsilon\Phi(r)-u)$, where $\chi$ is a $C^1$-smooth function defined on $\mathbb{R}^2$, then we have
	\begin{equation}\label{wpwmki}
		\begin{aligned}
			&\int_{\Sigma}\chi(\Phi(r),\varepsilon\Phi(r)-u) u H_k({\kappa}-\varepsilon) d \mu\\
			=&\int_{\Sigma} \chi(\Phi(r),\varepsilon\Phi(r)-u)(\lambda'-\varepsilon u) H_{k-1}({\kappa}-\varepsilon) d \mu  \\
			&+(k-1)\int_{\Sigma}\sum_{j=1}^{n}A_{j}\chi(\Phi(r),\varepsilon\Phi(r)-u) H_{k-2;j}(\tilde{h})d \mu\\
			& +  \frac{1}{k C_{n}^k} \int_{\Sigma} \left(T_{k-1}\right)^{ij}(\tilde{h})\left(\partial_{1}\chi\nabla_{i}\Phi+\partial_{2}\chi\nabla_{i}(\varepsilon\Phi-u)\right) \nabla_j \Phi d \mu.
		\end{aligned}
	\end{equation}
\end{lemma}
\begin{proof}
	Combining (\ref{wp2}) and (\ref{wp1}) together, we have
	\begin{equation}\label{wp3}
		\begin{aligned}
			\frac{1}{kC_{n}^k}\nabla_{i}\left(\nabla_{j}\Phi \left(T_{k-1}\right)^{ij}(\tilde{h})\right)
			=&(\lambda'-\varepsilon u)H_{k-1}(\kappa-\varepsilon)-uH_{k}(\kappa-\varepsilon)\\
			&-\frac{n+1-k}{k(n-1)}\frac{C_{n-1}^{k-2}}{C_{n}^{k}}\sum_{j=1}^{n}H_{k-2;j}(\tilde{h})\operatorname{Ric}(e_j,\nu)\nabla_{j}\Phi.
		\end{aligned}
	\end{equation}
	Multiplying above equation by the function $\phi$, and integrating by parts, one obtain the desired result (\ref{wppmk}).
\end{proof}

Finally, we need the Heintze-Karcher type inequalities due to Li, Wei and Xu \cite{LWX25}.
\begin{proposition}[\cite{LWX25}]\label{hkp}
	 Let $\Omega$ be a bounded domain with smooth boundary $\Sigma=\partial\Omega$ in hyperbolic space $\mathbb{H}^{n+1}$. Assume that the mean curvature of $\Sigma$ satisfies $H_{1}(\kappa)>-1$, then
	\begin{equation}\label{hkf}
		\int_{\Sigma} \frac{\lambda'+u}{H_{1}(\kappa)+1} d \mu \geq (n+1) \int_{\Omega} \lambda' d v .
	\end{equation}
	 Equality holds in (\ref{hkf}) if and only if $\Sigma$ is umbilic.
\end{proposition}

\begin{proposition}[\cite{LWX25}]\label{space form hkp}
	Let $M^{n+1}=[0, \bar{r}) \times N$ $(n \geq 2)$ be a sub-static warped product manifold with metric $\bar{g}=d r^2+\lambda(r)^2 g_N$ and potential $\lambda^{\prime}(r)$, where $N$ is an $n$-dimensional closed manifold. Let $\Omega \subset M^{n+1}$ be a bounded domain with a connected, static-convex boundary $\Sigma=\partial \Omega$ on which $\lambda^{\prime}>0$. Let $\varepsilon$ be a real number such that the following inequality
	$$
	\left(\lambda^{\prime}-\varepsilon u\right)\left(H_1(\kappa)-\varepsilon\right)>0
	$$
	holds everywhere on $\Sigma$, where $u=\left\langle\lambda \partial_r, \nu\right\rangle$. Then
	\begin{equation}\label{space form hkineq}
		\int_{\Sigma} \frac{\lambda^{\prime}-\varepsilon u}{H_1(\kappa)-\varepsilon} d \mu \geq(n+1) \int_{\Omega} \lambda^{\prime} d v.
	\end{equation}
	Moreover, if inequality (\ref{static convex}) holds and strictly holds at some point in $\Sigma$, then $\Sigma$ is umbilic and of constant mean curvature when the equality holds in (\ref{space form hkineq}).
\end{proposition}

\section{Proof of Theorems \ref{warped product thm1} and \ref{space forem thm1}}\label{sec3}
After all the preparation work, we are ready to prove our main theorems.

\begin{proof}[Proof of Theorem \ref{warped product thm1}]
	(i) Since $({\kappa}-\varepsilon)\in{\Gamma_{k}^+}$, we recall from (\ref{nm1}) that
	\begin{equation}\label{nmij**}
		H_{i}({\kappa}-\varepsilon)H_{j-1}({\kappa}-\varepsilon) \geq H_{i-1}({\kappa}-\varepsilon)H_{j}({\kappa}-\varepsilon), \quad 1\leq i < j \leq k,
	\end{equation}
	where all equalities hold if and only if $\Sigma$ is umbilical.
	Multiplying (\ref{nmij**}) by $a_i$ and $b_j$, and summing over $i$ and $j$, we obtain
	\begin{equation}\label{snmij**}
		\begin{aligned}
			&\sum_{i=1}^{l-1} a_i (\Phi(r),\varepsilon\Phi(r)-u) H_i ({\kappa}-\varepsilon) \sum_{j=l}^k b_j (\Phi(r),\varepsilon\Phi(r)-u) H_{j-1}({\kappa}-\varepsilon) \\
			\geq &\sum_{i=1}^{l-1} a_i (\Phi(r),\varepsilon\Phi(r)-u) H_{i-1}({\kappa}-\varepsilon) \sum_{j=l}^k b_j (\Phi(r),\varepsilon\Phi(r)-u) H_{j}({\kappa}-\varepsilon).
		\end{aligned}
	\end{equation}
	By the assumption
	$$
	\sum_{i=1}^{l-1} a_i (\Phi(r),\varepsilon\Phi(r)-u) H_i ({\kappa}-\varepsilon) =\sum_{j=l}^k b_j (\Phi(r),\varepsilon\Phi(r)-u) H_{j}({\kappa}-\varepsilon) >0, \quad 2 \leq l \leq k \leq n,
	$$
	we infer from (\ref{snmij**}) that
	\begin{equation}\label{snmij1**}
		\sum_{j=l}^k b_j (\Phi(r),\varepsilon\Phi(r)-u) H_{j-1}({\kappa}-\varepsilon) \geq \sum_{i=1}^{l-1} a_i (\Phi(r),\varepsilon\Phi(r)-u) H_{i-1}({\kappa}-\varepsilon).
	\end{equation}
	On the other hand, by (\ref{nm1}) and \cite[(29)]{WX14}, we have, for each $ 1\leq i < j \leq k$, $1\leq p \leq n$,
	\begin{equation}\label{h;j}
		(j-1)H_{j-2;p} (\tilde{h}) H_{i}({\kappa}-\varepsilon) > (i-1) H_j ({\kappa}-\varepsilon) H_{i-2;p} (\tilde{h}),
	\end{equation}
	where $\tilde{h}_j^i=h_j^i- \varepsilon\delta_j^i$. Multiplying (\ref{h;j}) by $a_i$ and $b_j$ and summing over $i$ and $j$, we obtain
	\begin{equation}
		\begin{aligned}
			&\sum_{j=l}^k (j-1) b_j (\Phi(r),\varepsilon\Phi(r)-u) H_{j-2;p} (\tilde{h}) \sum_{i=1}^{l-1} a_i (\Phi(r),\varepsilon\Phi(r)-u) H_i ({\kappa}-\varepsilon) \\
			> &\sum_{j=l}^k b_j (\Phi(r),\varepsilon\Phi(r)-u) H_{j} ({\kappa}-\varepsilon) \sum_{i=1}^{l-1} (i-1) a_i (\Phi(r),\varepsilon\Phi(r)-u) H_{i-2;p} (\tilde{h}).
		\end{aligned}
	\end{equation}
	Hence
	\begin{equation}\label{h;j1}
		\begin{aligned}
			&\sum_{j=l}^k (j-1) b_j (\Phi(r),\varepsilon\Phi(r)-u) H_{j-2;p} (\tilde{h})
			> \sum_{i=1}^{l-1} (i-1) a_i (\Phi(r),\varepsilon\Phi(r)-u) H_{i-2;p} (\tilde{h}).
		\end{aligned}
	\end{equation}
	We infer from (\ref{wpwmki}), for each $i$ and $j$,
	\begin{equation*}
		\begin{aligned}
			&\int_{\Sigma} a_i(\Phi(r),\varepsilon\Phi(r)-u)\left((\lambda'-\varepsilon u) H_{i-1}({\kappa}-\varepsilon) -u H_i({\kappa}-\varepsilon) \right)d \mu  \\
			=&-(i-1)\int_{\Sigma}a_i(\Phi(r),\varepsilon\Phi(r)-u)\sum_{p=1}^{n}A_{p} H_{i-2;p}(\tilde{h})d \mu\\
			&-\frac{1}{i C_{n}^i} \int_{\Sigma} \left(T_{i-1}\right)^{pq}(\tilde{h})\left(\partial_{1}a_i \nabla_{p}\Phi+\partial_{2}a_i \nabla_{p}(\varepsilon\Phi-u)\right) \nabla_q \Phi d \mu ,
		\end{aligned}
	\end{equation*}
	and
	\begin{equation*}
		\begin{aligned}
			&\int_{\Sigma} b_j(\Phi(r),\varepsilon\Phi(r)-u)\left((\lambda'-\varepsilon u) H_{j-1}({\kappa}-\varepsilon) -u H_j({\kappa}-\varepsilon) \right)d \mu  \\
			=&-(j-1)\int_{\Sigma}b_j(\Phi(r),\varepsilon\Phi(r)-u)\sum_{p=1}^{n}A_{p} H_{j-2;p}(\tilde{h})d \mu\\
			&-\frac{1}{j C_{n}^j} \int_{\Sigma} \left(T_{j-1}\right)^{pq}(\tilde{h})\left(\partial_{1}b_j \nabla_{p}\Phi+\partial_{2}b_j \nabla_{p}(\varepsilon\Phi-u)\right) \nabla_q \Phi d \mu.
		\end{aligned}
	\end{equation*}
	Now at a fixed point on $\Sigma$, we take an orthogonal frame $\{e_{1}, \cdots,e_{n}\}$  such that $g_{ij}=\delta_{ij}$ and $h_{ij}=\kappa_{i}\delta_{ij}$. Recall the assumptions $\partial_{1}a_i, \partial_{2}b_j \leq 0$, $\partial_{2}a_i, \partial_{1}b_j \geq 0$ and $h_{ij}>\varepsilon g_{ij}$ when $\partial_{2}a_i$ or $\partial_{2}b_j $ does not vanish at some point for some $0\le i\le l-1$ or $l\le j\le k$. Using (\ref{hession phi}), (\ref{hession u}) and $(\kappa-\varepsilon)\in{\Gamma_{k}^+}$, we have
	\begin{equation}\label{partial1**}
		\begin{aligned}
			&\left(T_{i-1}\right)^{pq}(\tilde{h})\left(\partial_{1}a_i \nabla_{p}\Phi+\partial_{2}a_i \nabla_{p}(\varepsilon\Phi-u)\right) \nabla_q \Phi \\
			=&\left(T_{i-1}\right)^{pq}(\tilde{h})\left(\partial_{1}a_i \nabla_{p}\Phi-\partial_{2}a_i (h_{p}^{l}\nabla_{l}\Phi-\varepsilon\nabla_{p}\Phi)\right) \nabla_q \Phi\\
			=&\sum_{p=1}^{n}\left(T_{i-1}\right)^{pp}(\tilde{h})\left(\partial_{1}a_i -\partial_{2}a_i (\kappa_{p}-\varepsilon)\right) |\nabla_p \Phi|^2\leq 0,
		\end{aligned}
	\end{equation}
	and
	\begin{equation}\label{partial2**}
		\begin{aligned}
			&\left(T_{j-1}\right)^{pq}(\tilde{h})\left(\partial_{1}b_j \nabla_{p}\Phi+\partial_{2}b_j \nabla_{p}(\varepsilon\Phi-u)\right) \nabla_q \Phi \\
			=&\left(T_{j-1}\right)^{pq}(\tilde{h})\left(\partial_{1}b_j \nabla_{p}\Phi-\partial_{2}b_j (h_{p}^{l}\nabla_{l}\Phi-\varepsilon\nabla_{p}\Phi)\right) \nabla_q \Phi\\
			=&\sum_{p=1}^{n}\left(T_{j-1}\right)^{pp}(\tilde{h})\left(\partial_{1}b_j -\partial_{2}b_j (\kappa_{p}-\varepsilon)\right) |\nabla_p \Phi|^2\geq 0.
		\end{aligned}
	\end{equation}
	Thus
	\begin{equation}\label{ai}
		\begin{aligned}
			&\int_{\Sigma} a_i(\Phi(r),\varepsilon\Phi(r)-u)\left((\lambda'-\varepsilon u) H_{i-1}({\kappa}-\varepsilon) -u H_i({\kappa}-\varepsilon) \right)d \mu  \\
			\geq&-(i-1)\int_{\Sigma}a_i(\Phi(r),\varepsilon\Phi(r)-u)\sum_{p=1}^{n}A_{p} H_{i-2;p}(\tilde{h})d \mu,
		\end{aligned}
	\end{equation}
	and
	\begin{equation}\label{bj}
		\begin{aligned}
			&\int_{\Sigma} b_j(\Phi(r),\varepsilon\Phi(r)-u)\left((\lambda'-\varepsilon u) H_{j-1}({\kappa}-\varepsilon) -u H_j({\kappa}-\varepsilon) \right)d \mu  \\
			\leq&-(j-1)\int_{\Sigma}b_j(\Phi(r),\varepsilon\Phi(r)-u)\sum_{p=1}^{n}A_{p} H_{j-2;p}(\tilde{h})d \mu.
		\end{aligned}
	\end{equation}
	Summing (\ref{ai}) over $i$ and (\ref{bj}) over $j$, and  taking the difference, we get
	\begin{align*}
		0 =\int_{\Sigma} & \left( \sum_{j=l}^k b_j (\Phi(r),\varepsilon\Phi(r)-u) H_{j}({\kappa}-\varepsilon) -\sum_{i=1}^{l-1} a_i (\Phi(r),\varepsilon\Phi(r)-u) H_i ({\kappa}-\varepsilon) \right) u d\mu \\  \geq \int_{\Sigma}  &\left( \sum_{j=l}^k b_j (\Phi(r),\varepsilon\Phi(r)-u) H_{j-1}({\kappa}-\varepsilon)\right.\\
		&\left.-\sum_{i=1}^{l-1} a_i (\Phi(r),\varepsilon\Phi(r)-u) H_{i-1} ({\kappa}-\varepsilon) \right) (\lambda'-\varepsilon u) d\mu \\
		&+\int_{\Sigma} \sum_{p=1}^{n} A_p \left(\sum_{j=l}^{k} (j-1) b_j(\Phi(r),\varepsilon\Phi(r)-u) H_{j-2;p}(\tilde{h}) \right. \\ &\left. - \sum_{i=1}^{l-1} (i-1) a_i (\Phi(r),\varepsilon\Phi(r)-u) H_{i-2;p}(\tilde{h}) \right) d \mu\\
		\geq & 0,
	\end{align*}
	where in the last inequality we used (\ref{Aj}), (\ref{snmij1**}), (\ref{h;j1}) and $\lambda'-\varepsilon u > 0$ on $\Sigma$. We conclude that the equality holds in the Newton-MacLaurin inequality (\ref{nm1}), which implies that $\Sigma$ is totally umbilical. Moreover, thanks to (\ref{h;j1}), we have
	\begin{equation}
		A_j \equiv 0, \quad \forall 1 \leq j \leq n,
	\end{equation}
	which implies that the unit normal $\nu$ is parallel to the radial direction $\partial_{r}$ everywhere on $\Sigma$, and hence $\Sigma$ is a slice.
	
	(ii) The proof follows a similar structure as before, with minor adjustments required for the index values. Proceeding as above, we have
	\begin{equation}\label{snmij111**}
		\sum_{j=l}^k b_j (\Phi(r),\varepsilon\Phi(r)-u) H_{j+1}({\kappa}-\varepsilon) \leq \sum_{i=0}^{l-1} a_i (\Phi(r),\varepsilon\Phi(r)-u) H_{i+1}({\kappa}-\varepsilon),
	\end{equation}
	\begin{equation}\label{h;j1*}
		\begin{aligned}
			&\sum_{j=l}^k j b_j (\Phi(r),\varepsilon\Phi(r)-u) H_{j-1;p} (\tilde{h})
			> \sum_{i=0}^{l-1} i a_i (\Phi(r),\varepsilon\Phi(r)-u) H_{i-1;p} (\tilde{h}),
		\end{aligned}
	\end{equation}
	and
	\begin{align*}
		0 =\int_{\Sigma} &\left( \sum_{i=0}^{l-1} a_i (\Phi(r),\varepsilon\Phi(r)-u) H_i ({\kappa}-\varepsilon)\right. \\
		&\left.- \sum_{j=l}^k b_j (\Phi(r),\varepsilon\Phi(r)-u) H_{j}({\kappa}-\varepsilon) \right) (\lambda' - \varepsilon u) d\mu \\  \geq  \int_{\Sigma} & \left( \sum_{i=0}^{l-1} a_i (\Phi(r),\varepsilon\Phi(r)-u) H_{i+1} ({\kappa}-\varepsilon) -\sum_{j=l}^k b_j (\Phi(r),\varepsilon\Phi(r)-u) H_{j+1}({\kappa}-\varepsilon) \right) u d\mu \\
		+\int_{\Sigma} & \sum_{p=1}^{n} A_p \left(\sum_{j=l}^{k} j b_j(\Phi(r),\varepsilon\Phi(r)-u) H_{j-1;p}(\tilde{h}) \right.\\
		&\left.- \sum_{i=1}^{l-1} i a_i (\Phi(r),\varepsilon\Phi(r)-u) H_{i-1;p}(\tilde{h}) \right) d \mu
		\geq  0.
	\end{align*}
	Here, the last inequality follows from (\ref{spf}), (\ref{Aj}), (\ref{snmij111**}) and (\ref{h;j1*}). We finish the proof by examining the equality case as before.
	
\end{proof}

\begin{proof}[Proof of Theorem \ref{space forem thm1}]
	For the space forms, we notice that $\tilde{h}_{ij}=h_{ij}-\varepsilon g_{ij}$ is a Codazzi tensor, i.e., $\nabla_{\ell}\tilde{h}_{ij}$ is symmetric in $i,j,\ell$. It follows that the Newton tensor $T_{k-1}(\tilde{h}
	)$ is divergence free, i.e., $\nabla_{i}\left(T_{k-1}^{ij}(\tilde{h})\right)=0$. Thus, (\ref{wp2}) implies
	\begin{equation}\label{sftk}
		\sum_{i,j=1}^{n}\nabla_{i}\left( T_{k-1}^{ij}(\tilde{h})\nabla_{j}\Phi\right)=(\lambda'-\varepsilon u)(n+1-k)\sigma_{k-1}(\kappa-\varepsilon)-uk\sigma_{k}(\kappa-\varepsilon).
	\end{equation}
	Integrating above equation, we have the Minkowski formula in the space forms
	\begin{equation}\label{sfmk}
		\int_{\Sigma}(\lambda'-\varepsilon u)H_{k-1}(\kappa-\varepsilon)d\mu=\int_{\Sigma}uH_{k}(\kappa-\varepsilon)d\mu.
	\end{equation}
	Multiplying (\ref{sftk}) by the function $\chi$ and integrating by parts, we obtain the weighted Minkowski formula in the space forms
	\begin{equation}\label{wmki}
	\begin{aligned}
	\int_{\Sigma}\chi(\Phi(r),\varepsilon\Phi(r)-u) u H_k({\kappa}&-\varepsilon) d \mu=\int_{\Sigma} \chi(\Phi(r),\varepsilon\Phi(r)-u)(\lambda'-\varepsilon u) H_{k-1}({\kappa}-\varepsilon) d \mu  \\
	&+\frac{1}{k C_{n}^k} \int_{\Sigma} \left(T_{k-1}\right)^{ij}(\tilde{h})\left(\partial_{1}\chi\nabla_{i}\Phi+\partial_{2}\chi\nabla_{i}(\varepsilon\Phi-u)\right) \nabla_j \Phi d \mu .
	\end{aligned}
	\end{equation}
	The remainder of the proof is almost the same as that of Theorem \ref{warped product thm1}. The only difference is that instead of (\ref{wpwmki}), we apply (\ref{wmki}) to complete the proof.

\end{proof}
	
\begin{proof}[Proof of Corollary \ref{coro3}]
    It follows from (\ref{pluscondition1}), (\ref{sfmk}) and (\ref{wmki}) that
		\begin{align*}
			\int_{\Sigma}H_{k-2}({\kappa}-\varepsilon)(\lambda'-\varepsilon u)d\mu =&\int_{\Sigma}H_{k-1}({\kappa}-\varepsilon)ud\mu \geq\int_{\Sigma}\chi(\Phi(r),\varepsilon\Phi(r)-u) u H_k({\kappa}-\varepsilon)d\mu \\ =&\int_{\Sigma} \chi(\Phi(r),\varepsilon\Phi(r)-u)(\lambda'-\varepsilon u) H_{k-1}({\kappa}-\varepsilon) d \mu  \\
			&+\frac{1}{k C_{n}^k} \int_{\Sigma} T_{k-1}^{ij}(\tilde{h})\left(\partial_{1}\chi\nabla_{i}\Phi+\partial_{2}\chi\nabla_{i}(\varepsilon\Phi-u)\right) \nabla_j \Phi d \mu \\
			\geq & \int_{\Sigma} \chi(\Phi(r),\varepsilon\Phi(r)-u)(\lambda'-\varepsilon u) H_{k-1}({\kappa}-\varepsilon) d \mu \\
			\geq & \int_{\Sigma}H_{k-2}({\kappa}-\varepsilon)(\lambda'-\varepsilon u)d\mu,
		\end{align*}
	then
	\begin{equation*}
		\int_{\Sigma}(H_{k-1}({\kappa}-\varepsilon) - \chi(\Phi(r),\varepsilon\Phi(r)-u) H_k({\kappa}-\varepsilon))ud\mu=0.
	\end{equation*}
	Combining with $H_{k-1}({\kappa}-\varepsilon) - \chi(\Phi(r),\varepsilon\Phi(r)-u) H_k({\kappa}-\varepsilon)\geq 0$ and $u>0$, we obtain $H_{k-1}({\kappa}-\varepsilon) = \chi(\Phi(r),\varepsilon\Phi(r)-u) H_k({\kappa}-\varepsilon)$. By Theorem \ref{space forem thm1}, we know that $\Sigma$ is a geodesic sphere.
\end{proof}

\section{Proof of Theorems \ref{warped product thm+2} and \ref{thm+2}}\label{sec4}
In this section, we will use the similar idea as in the work of Kwong, Lee and Pyo \cite{KLP18}. The main tools are Minkowski formulae as well as the Heintze-Karcher type inequality.

\begin{proof}[Proof of Theorem \ref{warped product thm+2}]
	Divided (\ref{bjhj*}) by $\eta$, it suffices to prove the result in the case where
	\begin{equation*}
		\begin{aligned}
			& \sum_{j=1}^{k}\bigg(a_j (\Phi(r),\varepsilon\Phi(r)-u) H_j ({\kappa}-\varepsilon) + b_j (\Phi(r),\varepsilon\Phi(r)-u) H_1 ({\kappa}-\varepsilon) H_{j-1} ({\kappa}-\varepsilon)\bigg) = 1.
		\end{aligned}
	\end{equation*}
	We infer from (\ref{Aj}) and (\ref{wpwmki}), for each $j$,
	\begin{equation*}
		\begin{aligned}
			&\int_{\Sigma} a_j (\Phi(r),\varepsilon\Phi(r)-u)\left((\lambda'-\varepsilon u) H_{j-1}({\kappa}-\varepsilon) -u H_j({\kappa}-\varepsilon) \right)d \mu  \\
			\leq &-\frac{1}{j C_{n}^j} \int_{\Sigma} \left(T_{j-1}\right)^{pq}(\tilde{h})\left(\partial_{1}a_j \nabla_{p}\Phi+\partial_{2}a_j \nabla_{p}(\varepsilon\Phi-u)\right) \nabla_q \Phi d \mu ,
		\end{aligned}
	\end{equation*}
	and
	\begin{equation*}
		\begin{aligned}
			&\int_{\Sigma} b_j (\Phi(r),\varepsilon\Phi(r)-u)\left((\lambda'-\varepsilon u) H_{j-1}({\kappa}-\varepsilon) -u H_j({\kappa}-\varepsilon) \right)d \mu  \\
			\leq &-\frac{1}{j C_{n}^j} \int_{\Sigma} \left(T_{j-1}\right)^{pq}(\tilde{h})\left(\partial_{1}b_j \nabla_{p}\Phi+\partial_{2}b_j \nabla_{p}(\varepsilon\Phi-u)\right) \nabla_q \Phi d \mu.
		\end{aligned}
	\end{equation*}
	Using (\ref{hession phi}), (\ref{hession u}) and the fact $(\kappa-\varepsilon) \in \Gamma_{k}^{+}$, we have
	\begin{equation}\label{partial1*}
		\begin{aligned}
			&\left(T_{j-1}\right)^{pq}(\tilde{h})\left(\partial_{1}a_j \nabla_{p}\Phi+\partial_{2}a_j \nabla_{p}(\varepsilon\Phi-u)\right) \nabla_q \Phi \\
			=&\left(T_{j-1}\right)^{pq}(\tilde{h})\left(\partial_{1}a_j \nabla_{p}\Phi-\partial_{2}a_j (h_{p}^{l}\nabla_{l}\Phi-\varepsilon\nabla_{p}\Phi)\right) \nabla_q \Phi\\
			=&\sum_{p=1}^{n}\left(T_{j-1}\right)^{pp}(\tilde{h})\left(\partial_{1}a_j -\partial_{2}a_j (\kappa_{p}-\varepsilon)\right) |\nabla_p \Phi|^2\geq 0.
		\end{aligned}
	\end{equation}
	Similarly,
	\begin{equation}\label{partial2*}
		\begin{aligned}
			\left(T_{j-1}\right)^{pq}(\tilde{h})\left(\partial_{1}b_j \nabla_{p}\Phi+\partial_{2}b_j \nabla_{p}(\varepsilon\Phi-u)\right) \nabla_q \Phi \geq 0.
		\end{aligned}
	\end{equation}
	Thus we have
	\begin{equation}\label{wpaj}
		\begin{aligned}
			&\int_{\Sigma} a_j (\Phi(r),\varepsilon\Phi(r)-u)\left((\lambda'-\varepsilon u) H_{j-1}({\kappa}-\varepsilon) -u H_j({\kappa}-\varepsilon) \right)d \mu \leq 0,
		\end{aligned}
	\end{equation}
	and
	\begin{equation}\label{wpbj}
		\begin{aligned}
			&\int_{\Sigma} b_j (\Phi(r),\varepsilon\Phi(r)-u)\left((\lambda'-\varepsilon u) H_{j-1}({\kappa}-\varepsilon) -u H_j({\kappa}-\varepsilon) \right)d \mu \leq 0.
		\end{aligned}
	\end{equation}
	Making use of (\ref{nm1}), (\ref{wpaj}) and (\ref{wpbj}), we derive
		\begin{align*}
		    \int_{\Sigma} u d \mu
		   = \int_{\Sigma} & u \bigg(\sum_{j=1}^{k}\left(a_j (\Phi(r),\varepsilon\Phi(r)-u) H_j ({\kappa}-\varepsilon) \right.\\
		   &\left.+ b_j (\Phi(r),\varepsilon\Phi(r)-u) H_1 ({\kappa}-\varepsilon) H_{j-1} ({\kappa}-\varepsilon)\right)\bigg) d \mu \\
		   \geq \int_{\Sigma} & u \bigg(\sum_{j=1}^{k}\left(a_j (\Phi(r),\varepsilon\Phi(r)-u) H_j ({\kappa}-\varepsilon) + b_j (\Phi(r),\varepsilon\Phi(r)-u) H_{j} ({\kappa}-\varepsilon)\right)\bigg) d \mu \\
		   \geq \int_{\Sigma} & (\lambda'-\varepsilon u) \bigg(\sum_{j=1}^{k}\left(a_j (\Phi(r),\varepsilon\Phi(r)-u) H_{j-1} ({\kappa}-\varepsilon) \right.\\
		   & \left.+ b_j (\Phi(r),\varepsilon\Phi(r)-u) H_{j-1} ({\kappa}-\varepsilon)\right)\bigg) d \mu \\
		   \geq \int_{\Sigma} & \frac{\lambda'-\varepsilon u}{H_1 (\kappa-\varepsilon)} \bigg(\sum_{j=1}^{k}\left(a_j (\Phi(r),\varepsilon\Phi(r)-u) H_{j} ({\kappa}-\varepsilon) \right.\\
		   &\left.+ b_j (\Phi(r),\varepsilon\Phi(r)-u) H_1 ({\kappa}-\varepsilon) H_{j-1} ({\kappa}-\varepsilon)\right)\bigg) d \mu \\
		   = \int_{\Sigma} & \frac{\lambda'-\varepsilon u}{H_1 (\kappa-\varepsilon)} d \mu.
		\end{align*}
	
	On the other hand, Li-Wei-Xu's inequality (Proposition \ref{space form hkp}) reduces a reverse inequality
	$$
	\int_{\Sigma} u d \mu= (n+1) \int_{\Omega} \lambda' d v \leq \int_{\Sigma} \frac{\lambda'-\varepsilon u}{H_1 (\kappa-\varepsilon)} d \mu.
	$$
	These two inequalities yield that the equalities hold in Li-Wei-Xu's inequality and Newton-Maclaurin inequalities. We conclude that $\Sigma$ is umbilic. Moreover, thanks to (\ref{wpaj}) and (\ref{wpbj}), we have
	\begin{equation*}
		A_j \equiv 0, \quad \forall 1 \leq j \leq n,
	\end{equation*}
	which implies that the unit normal $\nu$ is parallel to the radial direction $\partial_{r}$ everywhere on $\Sigma$, and hence $\Sigma$ is a slice.
	
\end{proof}

\begin{proof}[Proof of Theorem \ref{thm+2}]
	Since there exists at least one elliptic point on $\Sigma$ and the assumption $H_k (\kappa+1) > 0$, it follows from \cite[Proposition 2.1]{LWX25} that $(\kappa+1)\in{\Gamma_{k}^+}$. The remainder of the proof is almost the same as that of Theorem \ref{warped product thm+2}. The only difference is that instead of (\ref{wpwmki}) and Proposition \ref{space form hkp}, we apply (\ref{wmki}) and Proposition \ref{hkp}, respectively. We then complete the proof.
	
\end{proof}

\section{Proof of Theorems \ref{thm+3} and \ref{thm+4}}\label{sec5}
In this section, we establish two rigidity theorems for hypersurfaces satisfying fully nonlinear curvature conditions in warped product manifolds. The proofs combine geometric inequalities with careful analysis of the equality cases.

\begin{proof}[Proof of Theorem \ref{thm+3}]
	The assumption $({\kappa}-\varepsilon)\in\Gamma_{k}^+$ implies $H_j ({\kappa}-\varepsilon)>0$ for all $j=1,\cdots,k$. From equation (\ref{aij}), we immediately deduce $u>0$ on $\Sigma$.
	
	For the case of $k \geq 2$, applying the Newton-Maclaurin inequality (\ref{nm1}), we obtain for any $0 \leq i < j \leq k$,
	\begin{equation}\label{gnm}
		\left(\frac{1}{H_1 ({\kappa}-\varepsilon)}\right)^{j-i} \leq \frac{H_i ({\kappa}-\varepsilon)}{H_j ({\kappa}-\varepsilon)}=\prod_{m=i}^{j-1}\frac{H_m ({\kappa}-\varepsilon)}{H_{m+1} ({\kappa}-\varepsilon)} \leq \left(\frac{H_{j-1} ({\kappa}-\varepsilon)}{H_j ({\kappa}-\varepsilon)}\right)^{j-i}.
	\end{equation}
	This leads to two key estimates,
	\begin{equation}\label{u/v-u1}
		\begin{aligned}
			\beta \frac{u}{\lambda' - \varepsilon u}=&\sum_{i < j} a_{i,j}\left(\frac{H_{i}({\kappa}-\varepsilon)}{H_{j}({\kappa}-\varepsilon)}\right)^{\frac{1}{j-i}} \leq \sum_{i < j} a_{i,j}\frac{H_{j-1} ({\kappa}-\varepsilon)}{H_j ({\kappa}-\varepsilon)}\\
			\leq &\sum_{i < j} a_{i,j}\frac{H_{k-1} ({\kappa}-\varepsilon)}{H_k ({\kappa}-\varepsilon)}=\frac{H_{k-1} ({\kappa}-\varepsilon)}{H_k ({\kappa}-\varepsilon)},
		\end{aligned}
	\end{equation}
	and
	\begin{equation}\label{u/v-u2}
		\beta \frac{u}{\lambda' - \varepsilon u}=\sum_{i < j} a_{i,j}\left(\frac{H_{i}({\kappa}-\varepsilon)}{H_{j}({\kappa}-\varepsilon)}\right)^{\frac{1}{j-i}} \geq \sum_{i < j} a_{i,j}\frac{1}{H_1 ({\kappa}-\varepsilon)}=\frac{1}{H_1 ({\kappa}-\varepsilon)}.
	\end{equation}
	From (\ref{u/v-u1}) we derive
	\begin{equation}\label{hk}
		\beta \int_{\Sigma}uH_k ({\kappa}-\varepsilon)d\mu \leq \int_{\Sigma}(\lambda' - \varepsilon u)H_{k-1} ({\kappa}-\varepsilon)d\mu,
	\end{equation}
	which implies $\beta \leq 1$ by Lemma \ref{wpmkl}. While (\ref{u/v-u2}) gives
	\begin{equation*}
		\beta \int_{\Sigma}uH_1 ({\kappa}-\varepsilon)d\mu \geq \int_{\Sigma}(\lambda' - \varepsilon u)d\mu,
	\end{equation*}
	and thus $\beta \geq 1$ by (\ref{h1}). Therefore $\beta = 1$ and all inequalities in (\ref{gnm}) become equalities, which reduces $\Sigma$ being umbilical. The equality in (\ref{hk}) then implies that $\Sigma$ must be a slice.
	
	For the case of $k=1$, (\ref{u/v-u2}) becomes an equality, so $\beta = 1$ by (\ref{h1}). Using the Newton-Maclaurin inequality, we obtain
	\begin{equation}\label{gnm2}
		H_2 ({\kappa}-\varepsilon)\frac{u}{\lambda' - \varepsilon u}=\frac{H_{2}({\kappa}-\varepsilon)}{H_{1}({\kappa}-\varepsilon)} \leq \frac{H_{1}({\kappa}-\varepsilon)}{H_{0}({\kappa}-\varepsilon)}=H_{1}({\kappa}-\varepsilon).
	\end{equation}
	Integration \eqref{gnm2}, then combined with the Minkowski type formula (\ref{wpmkf}) for $k=2$ shows that (\ref{gnm2}) must be an equality. This completes the proof.
\end{proof}

\begin{proof}[Proof of Theorem \ref{thm+4}]
	The assumption $({\kappa}-\varepsilon)\in\Gamma_{k}^+$ says $H_k ({\kappa}-\varepsilon)>0$. It follows from (\ref{pluscondition2}) that $u>0$. We recall that Newton-Maclaurin's inequality (\ref{nm2})
	\begin{equation*}
		(H_{k}({\kappa}-\varepsilon))^{\frac{1}{k}}\leq (H_{k-1}({\kappa}-\varepsilon))^{\frac{1}{k-1}}
	\end{equation*}
	holds. We employ the Li-Wei-Xu's Heintze-Karcher type inequality (Proposition \ref{space form hkp})
	\begin{equation}\label{thm4hk}
		\int_{\Sigma}ud\mu = (n+1)\int_{\Omega} \lambda' d v \leq \int_{\Sigma}\frac{\lambda'-\varepsilon u}{H_{1}({\kappa}-\varepsilon)}d\mu,
	\end{equation}
    the Minkowski type formula (\ref{wpmkf}) and (\ref{Aj})
	\begin{equation}\label{thm4mk}
		\int_{\Sigma}\left( \lambda'- \varepsilon u \right) H_{k-1}({\kappa}-\varepsilon) d \mu \leq \int_{\Sigma} u H_k({\kappa}-\varepsilon) d \mu.
	\end{equation}
	Combining these with (\ref{pluscondition2}), we obtain
	\begin{equation}\label{holder1}
		\int_{\Sigma}(H_k({\kappa}-\varepsilon))^{-\alpha}(\lambda'-\varepsilon u)d\mu\leq \int_{\Sigma}\frac{\lambda'- \varepsilon u}{H_{1}({\kappa}-\varepsilon)}d\mu\leq\int_{\Sigma}(H_{k}({\kappa}-\varepsilon))^{-\frac{1}{k}}(\lambda'- \varepsilon u)d\mu
	\end{equation}
	and
	\begin{equation}\label{holder2}
		\int_{\Sigma}(H_k({\kappa}-\varepsilon))^{1-\alpha}(\lambda'-\varepsilon u)d\mu=\int_{\Sigma} u H_k({\kappa}-\varepsilon) d \mu \geq \int_{\Sigma}(\lambda'-\varepsilon u) H_{k-1}({\kappa}-\varepsilon) d \mu.
	\end{equation}
	Applying H\"{o}lder's inequality we have
	\begin{equation*}
		\begin{aligned}
			&\int_{\Sigma}(H_{k}({\kappa}-\varepsilon))^{-\frac{1}{k}}(\lambda'- \varepsilon u)\\
			\leq&\left(\int_{\Sigma}(H_{k-1}({\kappa}-\varepsilon))^{1-p}(H_{k}({\kappa}-\varepsilon))^{-\frac{p}{k}}(\lambda'- \varepsilon u)\right)^{\frac{1}{p}}\left(\int_{\Sigma}H_{k-1}({\kappa}-\varepsilon)(\lambda'- \varepsilon u)\right)^{\frac{p-1}{p}}\\
			\leq&\left(\int_{\Sigma}(H_{k}({\kappa}-\varepsilon))^{\frac{-kp-1+k}{k}}(\lambda'- \varepsilon u)\right)^{\frac{1}{p}}\left(\int_{\Sigma}(H_{k}({\kappa}-\varepsilon))^{1-\alpha}(\lambda'- \varepsilon u)\right)^{\frac{p-1}{p}}.
		\end{aligned}
	\end{equation*}
	Choose $p$, such that $\frac{-kp-1+k}{k}=-\alpha$, then $p=\frac{k\alpha+k-1}{k}$. The condition $p\geq 1$ implies $\alpha\geq\frac{1}{k}$. The above inequality becomes
	\begin{equation}\label{holder3}
		\begin{aligned}
			&\int_{\Sigma}(H_{k}({\kappa}-\varepsilon))^{-\frac{1}{k}}(\lambda'- \varepsilon u)\\
			\leq & \left(\int_{\Sigma}(H_{k}({\kappa}-\varepsilon))^{-\alpha}(\lambda'- \varepsilon u)\right)^{\frac{k}{k\alpha+k-1}}\left(\int_{\Sigma}(H_{k}({\kappa}-\varepsilon))^{1-\alpha}(\lambda'- \varepsilon
			 u)\right)^{\frac{k\alpha-1}{k\alpha+k-1}}.
		\end{aligned}
	\end{equation}
	Proceeding as above, we obtain
	\begin{equation*}
		\begin{aligned}
			&\int_{\Sigma}(H_k({\kappa}-\varepsilon))^{1-\alpha}(\lambda'- \varepsilon u)\\
			\leq & \left(\int_{\Sigma}(H_{k}({\kappa}-\varepsilon))^{\frac{-1+k+p-pk\alpha}{k}}(\lambda'- \varepsilon u)\right)^{\frac{1}{p}}\left(\int_{\Sigma}(H_{k}({\kappa}-\varepsilon))^{1-\alpha}(\lambda'- \varepsilon u)\right)^{\frac{p-1}{p}}.
		\end{aligned}
	\end{equation*}
	That is,
	\begin{equation*}
		\int_{\Sigma}(H_k({\kappa}-\varepsilon))^{1-\alpha}(\lambda'- \varepsilon u)\leq\int_{\Sigma}(H_{k}({\kappa}-\varepsilon))^{\frac{-1+k+p-pk\alpha}{k}}(\lambda'- \varepsilon u).
	\end{equation*}
	Choose $p$, such that $\frac{-1+k+p-pk\alpha}{k}=-\alpha$, then $p=\frac{k\alpha-1+k}{k\alpha-1}$. Thus, the above inequality becomes
	\begin{equation*}
		\int_{\Sigma}(H_k({\kappa}-\varepsilon))^{1-\alpha}(\lambda'- \varepsilon u)\leq\int_{\Sigma}(H_{k}({\kappa}-\varepsilon))^{-\alpha}(\lambda'- \varepsilon u).
	\end{equation*}
	Substituting the above inequality into (\ref{holder3}), we obtain
	\begin{equation}\label{holder4}
		\int_{\Sigma}(H_{k}({\kappa}-\varepsilon))^{-\frac{1}{k}}(\lambda'- \varepsilon u)\leq\int_{\Sigma}(H_{k}({\kappa}-\varepsilon))^{-\alpha}(\lambda'- \varepsilon u).
	\end{equation}
	Combining with (\ref{holder1}), it illustrates that the equalities hold in the Heintze-Karcher type inequality (\ref{thm4hk}) and the Minkowski type formula (\ref{thm4mk}). We conclude that $\Sigma$ is a slice as before.
	
\end{proof}

\section{Proof of Theorems \ref{glwthm} and \ref{glwthm1}}\label{sec6}
In this section, we investigate some rigidity results in warped product manifolds where the fiber does not have constant sectional curvature. Theorem \ref{glwthm} and \ref{glwthm1} will be proved using the classical integral method due to Montiel and Ros \cite{MR91,Ros87}. The main tools are minkowski formula established by \cite{GLW19} and Heintze-Karcher type inequality (\ref{space form hkineq}).

\begin{proof}[Proof of Theorem \ref{glwthm}]
	(i) Recall the assumptions $\partial_{1}a \geq 0$, $\partial_{2}a \leq 0$ and $h_{ij}>\varepsilon g_{ij}$ if $\partial_{2}a$ dose not vanish at some point. Using (\ref{hession phi}), (\ref{hession u}) and integrating by parts, we have
	\begin{equation}\label{glwa1}
		\begin{aligned}
			& \int_{\Sigma} a (\Phi(r),\varepsilon\Phi(r)-u) \left((\lambda' - \varepsilon u) - u H_1 (\kappa - \varepsilon)\right) d \mu \\
			= & \frac{1}{n} \int_{\Sigma} a (\Phi(r),\varepsilon\Phi(r)-u) \Delta \Phi d \mu \\
			= & -\frac{1}{n} \int_{\Sigma} \sum_{p=1}^{n} \left(\partial_{1}a -\partial_{2}a (\kappa_{p}-\varepsilon)\right) |\nabla_p \Phi|^2 d \mu \leq 0.
		\end{aligned}
	\end{equation}
	Applying the above formula, we obtain
	\begin{align*}
		& a (\Phi(r),\varepsilon\Phi(r)-u) H_1 (\kappa - \varepsilon) \int_{\Sigma} u d\mu \\
		= & \int_{\Sigma} u a (\Phi(r),\varepsilon\Phi(r)-u) H_1 (\kappa - \varepsilon) d\mu\\
		\geq & \int_{\Sigma} (\lambda' - \varepsilon u) a (\Phi(r),\varepsilon\Phi(r)-u) d\mu \\
		= & a (\Phi(r),\varepsilon\Phi(r)-u) H_1 (\kappa - \varepsilon) \int_{\Sigma} \frac{\lambda' - \varepsilon u}{H_1 (\kappa - \varepsilon)} d\mu.
	\end{align*}
	Thus we have
	\begin{equation}\label{reverse hk}
		\int_{\Sigma} u d\mu \geq \int_{\Sigma} \frac{\lambda' - \varepsilon u}{H_1 (\kappa - \varepsilon)} d\mu.
	\end{equation}
	
	On the other hand, by Proposition \ref{space form hkp}, we derive that
	$$
	\int_{\Sigma} u d \mu= (n+1) \int_{\Omega} \lambda' d v \leq \int_{\Sigma} \frac{\lambda'-\varepsilon u}{H_1 (\kappa-\varepsilon)} d \mu.
	$$
	The latter inequality, coupled with (\ref{reverse hk}), gives that equality holds in (\ref{space form hkineq}), which implies that $\Sigma$ is totally umbilical. As in the proof of \cite[Sec. 4]{B13}, the condition (\ref{glwcondition}) implies $\Sigma$ is a slice.
	
	(ii) Denote by $\tilde{h}_j^i=h_j^i- \varepsilon\delta_j^i$, applying (\ref{hession phi}), we have
	\begin{equation}\label{intt1}
		\begin{aligned}
			\nabla_j \left( T_1^{ij}(\tilde{h})\nabla_i \Phi\right) = n(n-1) \left((\lambda' - \varepsilon u) H_1 (\kappa - \varepsilon) - u H_2 (\kappa - \varepsilon)\right) + \nabla_i \Phi \nabla_j (T_1^{ij}(\tilde{h})).
		\end{aligned}
	\end{equation}
	As in the proof of \cite[Lemma 2.3]{GLW19} and \cite[Lemma 3.1]{GLW19}, we get
	\begin{equation}\label{divt1}
		\begin{aligned}
			\nabla_i \Phi \nabla_j (T_1^{ij}(\tilde{h}))
			=& -g^{ij} \operatorname{Ric}(e_i, \nu) \nabla_j \Phi \\
			=& (n-1) \left(\frac{\lambda^{\prime\prime}}{\lambda} + \frac{K - {\lambda^{\prime}}^2}{\lambda^2}\right) \frac{u}{\lambda^2} |\nabla \Phi|^2 \\
			& + g^{ij} \left(({\operatorname{Ric}_N})_{ik} - (n-1) K ({g_N})_{ik}\right) \frac{u}{\lambda^2} \nabla^k r \nabla_j r \\
			\geq & 0,
		\end{aligned}
	\end{equation}
	where in the last inequality we used the condition (\ref{glwcondition}). Multiplying (\ref{intt1}) by $a$ and integrating by parts, by (\ref{divt1}), we can get
	\begin{equation}\label{glwa2}
		\begin{aligned}
			& \int_{\Sigma} a (\Phi(r),\varepsilon\Phi(r)-u) \left((\lambda' - \varepsilon u) H_1 (\kappa - \varepsilon) - u H_2 (\kappa - \varepsilon)\right) d \mu \\
			\leq & \frac{1}{n(n-1)} \int_{\Sigma} a (\Phi(r),\varepsilon\Phi(r)-u) \nabla_j \left( T_1^{ij}(\tilde{h})\nabla_i \Phi\right) d \mu \\
			= & -\frac{1}{n(n-1)} \int_{\Sigma} \sum_{p=1}^{n} T_1^{pp}(\tilde{h}) \left(\partial_{1}a -\partial_{2}a (\kappa_{p}-\varepsilon)\right) |\nabla_p \Phi|^2 d \mu \leq 0.
		\end{aligned}
	\end{equation}
	Applying the above formula and Newton-MacLaurin inequality (\ref{nm1}), we obtain
	\begin{align*}
		& a (\Phi(r),\varepsilon\Phi(r)-u) H_2 (\kappa - \varepsilon) \int_{\Sigma} u d\mu \\
		= & \int_{\Sigma} u a (\Phi(r),\varepsilon\Phi(r)-u) H_2 (\kappa - \varepsilon) d\mu\\
		\geq & \int_{\Sigma} (\lambda' - \varepsilon u) a (\Phi(r),\varepsilon\Phi(r)-u) H_1 (\kappa - \varepsilon) d\mu \\
		= & a (\Phi(r),\varepsilon\Phi(r)-u) H_2 (\kappa - \varepsilon) \int_{\Sigma} \frac{(\lambda' - \varepsilon u)H_1 (\kappa - \varepsilon)}{H_2 (\kappa - \varepsilon)} d\mu\\
		\geq & a (\Phi(r),\varepsilon\Phi(r)-u) H_2 (\kappa - \varepsilon) \int_{\Sigma} \frac{\lambda' - \varepsilon u}{H_1 (\kappa - \varepsilon)} d\mu.
	\end{align*}
	Thus we have
	\begin{equation}\label{reverse hk1}
		\int_{\Sigma} u d\mu \geq \int_{\Sigma} \frac{\lambda' - \varepsilon u}{H_1 (\kappa - \varepsilon)} d\mu.
	\end{equation}
	
	On the other hand, by Proposition \ref{space form hkp}, we derive that
	$$
	\int_{\Sigma} u d \mu= (n+1) \int_{\Omega} \lambda' d v \leq \int_{\Sigma} \frac{\lambda'-\varepsilon u}{H_1 (\kappa-\varepsilon)} d \mu.
	$$
	As before, the latter inequality, coupled with (\ref{reverse hk1}), implies that  $\Sigma$ is totally umbilical and then is a slice.
	
	(iii) Assume that $a (\Phi(r),\varepsilon\Phi(r)-u)\frac{H_2 ({\kappa}-\varepsilon)}{H_1 ({\kappa}-\varepsilon)} = c >0$ for some constant $c$. Applying the Newton-MacLaurin inequality (\ref{nm1}), we note that
	\begin{equation}\label{ratio ineq}
		a (\Phi(r),\varepsilon\Phi(r)-u) H_1 ({\kappa}-\varepsilon) \geq a (\Phi(r),\varepsilon\Phi(r)-u)\frac{H_2 ({\kappa}-\varepsilon)}{H_1 ({\kappa}-\varepsilon)} = c.
	\end{equation}
	It follows from (\ref{glwa2}) and (\ref{h1}) that
	\begin{equation*}
		\begin{aligned}
			&\int_{\Sigma} (\lambda' - \varepsilon u) a (\Phi(r),\varepsilon\Phi(r)-u) H_1 ({\kappa}-\varepsilon) d \mu \\
			\leq&\int_{\Sigma} u a (\Phi(r),\varepsilon\Phi(r)-u) H_2 ({\kappa}-\varepsilon) d \mu \\
			=&c\int_{\Sigma} u H_1 ({\kappa}-\varepsilon) d \mu = c \int_{\Sigma} (\lambda' - \varepsilon u) d \mu.
		\end{aligned}
	\end{equation*}
	This gives
	\begin{equation*}
		\int_{\Sigma} (\lambda' - \varepsilon u) \left(a (\Phi(r),\varepsilon\Phi(r)-u) H_1 ({\kappa}-\varepsilon) - c\right) d \mu \leq 0.
	\end{equation*}
	The above together with (\ref{ratio ineq}) imply that $a (\Phi(r),\varepsilon\Phi(r)-u) H_1 ({\kappa}-\varepsilon) = c$. Finally, from Theorem \ref{glwthm}(i), we complete the proof.
	
	(iv) By dividing (\ref{glw condition}) by $\eta$, it suffices to prove the result in the case where
	\begin{equation*}
		\begin{aligned}
			& a(\Phi(r),\varepsilon\Phi(r)-u) H_2 ({\kappa}-\varepsilon)
			+ b(\Phi(r),\varepsilon\Phi(r)-u) (H_1 ({\kappa}-\varepsilon))^2 = 1.
		\end{aligned}
	\end{equation*}
	Repeating the argument used to establish (\ref{glwa2}), we can also get
	\begin{align*}
		& \int_{\Sigma} b (\Phi(r),\varepsilon\Phi(r)-u) \left((\lambda' - \varepsilon u) H_1 (\kappa - \varepsilon) - u H_2 (\kappa - \varepsilon)\right) d \mu \\
		\leq & -\frac{1}{n(n-1)} \int_{\Sigma} \sum_{p=1}^{n} T_1^{pp}(\tilde{h}) \left(\partial_{1}b -\partial_{2}b (\kappa_{p}-\varepsilon)\right) |\nabla_p \Phi|^2 d \mu \leq 0.
	\end{align*}
	Therefore, by Newton-MacLaurin inequality (\ref{nm1}),
	\begin{align*}
		\int_{\Sigma} u d \mu = \int_{\Sigma} & u \left(a(\Phi(r),\varepsilon\Phi(r)-u) H_2 ({\kappa}-\varepsilon)
		+ b(\Phi(r),\varepsilon\Phi(r)-u) (H_1 ({\kappa}-\varepsilon))^2\right) d \mu\\
		\geq \int_{\Sigma} & u
		\left(a(\Phi(r),\varepsilon\Phi(r)-u) H_2 ({\kappa}-\varepsilon) + b(\Phi(r),\varepsilon\Phi(r)-u)
		H_2 ({\kappa}-\varepsilon)\right) d \mu\\
		\geq \int_{\Sigma} & (\lambda' - \varepsilon u) \left(a (\Phi(r),\varepsilon\Phi(r)-u)H_1 ({\kappa}-\varepsilon)+ b(\Phi(r),\varepsilon\Phi(r)-u)H_1 ({\kappa}-\varepsilon)\right) d \mu \\
		\geq \int_{\Sigma} & \frac{\lambda' - \varepsilon u}{H_1 ({\kappa}-\varepsilon)}\left(a (\Phi(r),\varepsilon\Phi(r)-u)H_2 ({\kappa}-\varepsilon)\right.\\
		&\left.+b(\Phi(r),\varepsilon\Phi(r)-u) (H_1 ({\kappa}-\varepsilon))^2\right)  d \mu \\
		= \int_{\Sigma} & \frac{\lambda' - \varepsilon u}{H_1 ({\kappa}-\varepsilon)} d \mu
		\geq (n+1) \int_{\Omega}\lambda'dv= \int_{\Sigma} u d\mu.
	\end{align*}
	We finish the proof by examining the equality case as before.
	
\end{proof}

\begin{proof}[Proof of Theorem \ref{glwthm1}]
	Since $\lambda'>0$, there exists at least one elliptic point on $\Sigma$. This together with the fact $a(\Phi(r),\varepsilon\Phi(r)-u)>0$ yields that the constant is positive and thus $H_1 ({\kappa})>0$ and $H_2 ({\kappa})>0$. The remainder of the proof is almost the same as that of Theorem \ref{glwthm}. The only difference is that instead of (\ref{space form hkineq}), we apply Brendle's Heintze-Karcher inequality \cite{B13}. Thus we complete the proof.
	
\end{proof}

{\bf Acknowledgements.}
Sheng was partially supported by National Key R$\&$D Program of China(No. 2022YFA1005500) and  Natural Science Foundation of China under Grant No. 12031017. Wu was partially supported by National Key R$\&$D Program of China(No. 2022YFA1005501).


\vspace{5mm}




\begin{thebibliography}{99}	

    \bibitem{ADM13}L.J. Al\'{\i}as, D. Impera and M. Rigoli, \textit{Hypersurfaces of constant higher order mean curvature in warped products}, Transactions of the American Mathematical Society, 2013, \textbf{365}(2): 591-621.

    \bibitem{A56}A.D. Alexsandorov, \textit{Uniqueness theorem for surfaces in the large I}, Vestnik Leningrad Univ., 1956, \textbf{11}: 5-17.

    \bibitem{B94}B. Andrews, \textit{Contraction of convex hypersurfaces in Riemannian spaces}, Journal of Differential Geometry, 1994, \textbf{39}(2): 407-431.

    \bibitem{BCW21}B. Andrews, X. Chen and Y. Wei, \textit{Volume preserving flow and Alexandrov–Fenchel type inequalities in hyperbolic space}, Journal of the European Mathematical Society, 2021, \textbf{23}(7): 2467-2509.

    \bibitem{AD14}C.P. Aquino and H.F. de Lima, \textit{On the unicity of complete hypersurfaces immersed in a semi-Riemannian warped product}, The Journal of Geometric Analysis, 2014, \textbf{24}: 1126-1143.

    \bibitem{BC97}J.L.M. Barbosa and A.G. Colares, \textit{Stability of hypersurfaces with constant $r$-mean curvature}, Annals of Global Analysis and Geometry, 1997, \textbf{15}(3): 277-297.

    \bibitem{B13}S. Brendle, \textit{Constant mean curvature surfaces in warped product manifolds}, Publications math\'{e}matiques de l'IH\'{E}S, 2013, \textbf{117}(1): 247-269.

    \bibitem{BE13}S. Brendle and M. Eichmair, \textit{Isoperimetric and Weingarten surfaces in the Schwarzschild manifold}, Journal of Differential Geometry, 2013, \textbf{94}(3): 387-407.

    \bibitem{BH17}S. Brendle and G. Huisken, \textit{A fully nonlinear flow for two-convex hypersurfaces in Riemannian manifolds}, Inventiones mathematicae, 2017, \textbf{210}: 559-613.

    \bibitem{BHW16}S. Brendle, P.K. Hung and M.T. Wang, \textit{A Minkowski Inequality for Hypersurfaces in the Anti‐de Sitter‐Schwarzschild Manifold}, Communications on Pure and Applied Mathematics, 2016, \textbf{69}(1): 124-144.

    \bibitem{G23}S. Gao, \textit{Closed self-similar solutions to flows by negative powers of curvature}, The Journal of Geometric Analysis, 2023, \textbf{33}(12): 370.

    \bibitem{GM21}S. Gao and H. Ma, \textit{Characterizations of umbilic hypersurfaces in warped product manifolds}, Frontiers of Mathematics in China, 2021, \textbf{16}: 689-703.

    \bibitem{G02}P. Guan, \textit{Topics in geometric fully nonlinear equations}, Lecture Notes, http://www. math. mcgill. ca/guan/notes. html, 2002.

    \bibitem{GLW19}P. Guan, J. Li and M.T. Wang, \textit{A volume preserving flow and the isoperimetric problem in warped product spaces}, Transactions of the American Mathematical Society, 2019, \textbf{372}(4): 2777-2798.

    \bibitem{HK78} E. Heintze and H. Karcher, \textit{A general comparison theorem with applications to volume estimates for submanifolds}, Annales Scientifiques de l'\'{E}cole Normale Sup\'{e}rieure. 1978, \textbf{11}(4): 451-470.

    \bibitem{HLMG09}Y. He, H. Li, H. Ma, and J. Ge, \textit{Compact embedded hypersurfaces with constant higher order anisotropic mean curvatures}, Indiana University Mathematics Journal, 2009, \textbf{58}(2): 853-868.

    \bibitem{HMZ01}O. Hijazi, S. Montiel and X. Zhang, \textit{Dirac Operator on Embedded Hypersurfaces}, Mathematical Research Letters, 2001, \textbf{8}(2): 195-208.

    \bibitem{HLW22}Y. Hu, H. Li and Y. Wei, \textit{Locally constrained curvature flows and geometric inequalities in hyperbolic space}, Mathematische Annalen, 2022, \textbf{382}(3-4): 1425-1474.

    \bibitem{HWZ23}Y. Hu, Y. Wei and T. Zhou, \textit{A Heintze-Karcher type inequality in hyperbolic space}, The Journal of Geometric Analysis, 2024, \textbf{34}(4): 113.

    \bibitem{K98}S.E. Koh, \textit{A characterization of round spheres}, Proceedings of the American Mathematical Society, 1998, \textbf{126}(12): 3657-3660.

    \bibitem{K00}S.E. Koh, \textit{Sphere theorem by means of the ratio of mean curvature functions}, Glasgow Mathematical Journal, 2000, \textbf{42}(1): 91-95.

    \bibitem{KLP18}K.K. Kwong, H. Lee and J. Pyo, \textit{Weighted Hsiung-Minkowski formulas and rigidity of umbilical hypersurfaces}, Mathematical Research Letters, 2018, \textbf{25}(2): 597-616.

    \bibitem{LWX14}H. Li, Y. Wei and C. Xiong, \textit{A note on Weingarten hypersurfaces in the warped product manifold}, International Journal of Mathematics, 2014, \textbf{25}(14): 1450121.

    \bibitem{LWX25}H. Li, Y. Wei and B. Xu, \textit{New Heintze-Karcher type inequalities in sub-static warped product manifolds}, arXiv preprint arXiv:2504.15109, 2025.

    \bibitem{LX} H. Li and B. Xu, \textit{Hyperbolic $p$-sum and horospherical $p$-Brunn-Minkowski theory in hyperbolic space}, arXiv: 2211.06875, 2022.

    \bibitem{LX19}J. Li and C. Xia, \textit{An integral formula and its applications on sub-static manifolds}, Journal of Differential Geometry, 2019, \textbf{113}(3): 493-518.

    \bibitem{M99}S. Montiel, \textit{Uniqueness of spacelike hypersurfaces of constant mean curvature in foliated spacetimes}, Mathematische Annalen, 1999, \textbf{314}(3): 529-553.

    \bibitem{MR91}S. Montiel and A. Ros, \textit{Compact hypersurfaces: the Alexandrov theorem for higher order mean curvatures}, Pitman Monographs and Surveys in Pure and Applied Mathematics \textbf{52} (1991) (in honor of M.P. do Carmo; edited by B. Lawson and K. Tenenblat), 279-296.

    \bibitem{R74}R. Reilly, \textit{On the Hessian of a function and the curvatures of its graph}, Michigan Mathematical Journal, 1974, \textbf{20}(4): 373-383.

    \bibitem{R77}R. Reilly, \textit{Applications of the Hessian operator in a Riemannian manifold}, Indiana University Mathematics Journal, 1977, \textbf{26}(3): 459-472.

    \bibitem{Ros88}A. Ros, \textit{Compact hypersurfaces with constant scalar curvature and a congruence theorem}, Journal of Differential Geometry, 1988, \textbf{27}(2): 215-220.

    \bibitem{Ros87}A. Ros, \textit{Compact hypersurfaces with constant higher order mean curvatures}, Revista Matem\'{a}tica Iberoamericana, 1987, \textbf{3}(3): 447-453.

    \bibitem{S60}R. Stong, \textit{Some characterizations of Riemann $n$-spheres}, Proceedings of the American Mathematical Society, 1960, \textbf{11}(6): 945-951.

    \bibitem{SWW24}W. Sheng, Y. Wang and J. Wu, \textit{On rigidity of hypersurfaces with constant shifted curvature functions in hyperbolic space}, Communications in Analysis and Geometry (to appear), arXiv:2402.04622.

    \bibitem{WWZ23}X. Wang, Y. Wei and T. Zhou, \textit{Shifted inverse curvature flows in hyperbolic space} Calculus of Variations and Partial Differential Equations, 2023, \textbf{62}(3): 93.

    \bibitem{W16}J. Wu, \textit{A new characterization of geodesic spheres in the hyperbolic space}, Proceedings of the American Mathematical Society, 2016, \textbf{144}(7): 3077-3084.

    \bibitem{WX14}J. Wu and C. Xia, \textit{On rigidity of hypersurfaces with constant curvature functions in warped product manifolds}, Annals of Global Analysis and Geometry, 2014, \textbf{46}(1): 1-22.
	
\end{thebibliography}
\end{document}